\documentclass{aims}
\usepackage{amsmath}
\usepackage{paralist}
\usepackage{graphics}
\usepackage{epsfig}
\usepackage{graphicx}
\usepackage{epstopdf}
\usepackage[colorlinks=true]{hyperref}
\hypersetup{urlcolor=blue, citecolor=red}

1

\def\currentvolume{X}
\def\currentissue{X}
\def\currentyear{200X}
\def\currentmonth{XX}
\def\ppages{X--XX}
\def\DOI{10.3934/xx.xx.xx.xx}

\newtheorem{Thm}{Theorem}
\newtheorem{Cor}[Thm]{Corollary}
\newtheorem{Lem}[Thm]{Lemma}
\newtheorem{Prop}[Thm]{Proposition}
\newtheorem{Rem}[Thm]{Remark}
\newtheorem*{Def}{Definition}

\newcommand{\C}{{\mathsf C}}
\newcommand{\R}{{\mathbb R}}
\renewcommand{\S}{{\mathbb S}}

\renewcommand{\L}{\mathrm L}

\newcommand{\be}[1]{\begin{equation}\label{#1}}
\newcommand{\ee}{\end{equation}}
\renewcommand{\(}{\left(}
\renewcommand{\)}{\right)}

\newcommand{\nrm}[2]{\|{#1}\|_{\L^{#2}(\R^d)}}

\newcommand{\LL}{\mathrm L}

\newcommand{\M}{\mathfrak M}
\newcommand{\iM}[1]{\int_{\M}{#1}\,d\kern1pt v_g}

\newcommand{\irdsph}[3]{\int_0^\infty\kern-5pt\int_{\S^{d-1}}#2\,{#1}^{\kern1pt #3}\,\frac{d{#1}}{{#1}}\,d\omega}
\newcommand{\D}[1]{\mathsf D_\alpha\kern1pt#1}

\newcommand{\rate}{\rho}
\newcommand{\dy}{\,dy}

\newcommand{\RR}{{\mathbb R}}
\newcommand{\dx}{\,dx}
\newcommand{\rd}{d}
\newcommand{\dt}{\,dt}

\newcommand{\FL}{\mathsf{F}} 
\newcommand{\I}{\mathcal{I}} 
\newcommand{\IL}{\mathsf{I}} 
\DeclareMathOperator*{\osc}{\rm osc}
\def\qed{\,\unskip\kern 6pt \penalty 500 \raise -2pt\hbox{\vrule \vbox to8pt{\hrule width 6pt \vfill\hrule}\vrule}\par}

\definecolor{darkblue}{rgb}{0.05, .05, .65}
\definecolor{darkgreen}{rgb}{0.1, .65, .1}
\definecolor{darkred}{rgb}{0.8,0,0}

\newcommand{\ttheta}{\vartheta}
\newcommand{\varth}{\theta}
\newcommand{\etanu}{\nu}
\newcommand{\nueta}{\eta}
\newcommand{\scaling}{\mu}

\begin{document}
\title[Weighted fast diffusion II]{Weighted fast diffusion equations (Part II):\\ Sharp asymptotic rates of convergence\\ in relative error by entropy methods}
\author[M.~Bonforte, J.~Dolbeault, M.~Muratori and B.~Nazaret]{}

\date{\today}

\subjclass{Primary: 35K55, 35B40, 49K30; Secondary: 26D10, 35B06, 46E35, 49K20, 35J20.}
\keywords{Fast diffusion equation, self-similar solutions, asymptotic behavior, intermediate asymptotics, rate of convergence, entropy methods, free energy, Caffarelli-Kohn-Nirenberg inequalities, Hardy-Poincar\'e inequalities, weights, optimal functions, best constants, symmetry breaking, linearization, spectral gap, Harnack inequality, parabolic regularity.}

\email{matteo.bonforte@uam.es}
\email{dolbeaul@ceremade.dauphine.fr}
\email{matteo.muratori@unipv.it}
\email{bruno.nazaret@univ-paris1.fr}

\thanks{$^*$ Corresponding author.}

\maketitle
\thispagestyle{empty}
\vspace*{-0.45cm}

\centerline{\scshape Matteo Bonforte}
\smallskip {\footnotesize
\centerline{Departamento de Matem\'{a}ticas,}
\centerline{Universidad Aut\'{o}noma de Madrid,}
\centerline{Campus de Cantoblanco, 28049 Madrid, Spain}
}\medskip

\centerline{\scshape Jean Dolbeault $^*$}
\smallskip {\footnotesize
\centerline{Ceremade, UMR CNRS n$^\circ$~7534,}
\centerline{Universit\'e Paris-Dauphine, PSL Research University,}
\centerline{Place de Lattre de Tassigny, 75775 Paris Cedex~16, France}
}\medskip

\centerline{\scshape Matteo Muratori}
\smallskip {\footnotesize
\centerline{Dipartimento di Matematica \emph{Felice Casorati},}
\centerline{Universit\`a degli Studi di Pavia,}
\centerline{Via A.~Ferrata 5, 27100 Pavia, Italy}
}\medskip

\centerline{\scshape Bruno Nazaret}
\smallskip {\footnotesize
\centerline{SAMM,}
\centerline{Universit\'e Paris 1,}
\centerline{90, rue de Tolbiac, 75634 Paris Cedex~13, France}
}

\begin{abstract}\vspace*{-0.25cm} This paper is the second part of the study. In Part~I, self-similar solutions of a weighted fast diffusion equation (WFD) were related to optimal functions in a family of subcritical Caffarelli-Kohn-Nirenberg inequalities (CKN) applied to radially symmetric functions. For these inequalities, the linear instability (symmetry breaking) of the optimal radial solutions relies on the spectral properties of the linearized evolution operator. Symmetry breaking in (CKN) was also related to large-time asymptotics of (WFD), at formal level. A first purpose of Part~II is to give a rigorous justification of this point, that is, to determine the asymptotic rates of convergence of the solutions to (WFD) in the symmetry range of (CKN) as well as in the symmetry breaking range, and even in regimes beyond the supercritical exponent in (CKN). Global rates of convergence with respect to a free energy (or entropy) functional are also investigated, as well as uniform convergence to self-similar solutions in the strong sense of the relative error. Differences with large-time asymptotics of fast diffusion equations without weights are emphasized.
\end{abstract}

\section{Introduction}\label{Sec:Intro}

Let us consider the \emph{fast diffusion equation with weights}
\be{FD}
u_t+|x|^\gamma\,\nabla\cdot\(\,|x|^{-\beta}\,u\,\nabla u^{m-1}\)=0\,,\quad(t,x)\in\R^+\times\R^d\,.
\ee
Such an equation admits the \emph{self-similar solution}
\[
u_\star(t,x)=\(\tfrac\rate t\)^{\rate\,(d-\gamma)}\,\(1+|(\tfrac\rate t)^\rate\,x|^{2+\beta-\gamma}\)^\frac1{m-1}\,,\quad\forall\,(t,x)\in\R^+\times\R^d\,,
\]
where $\tfrac1{\rate}=(d-\gamma)\,(m-m_c)$ with $m_c:=\tfrac{d-2-\beta}{d-\gamma}$. At least when $1-m>0$ is not too big, this self-similar solution attracts all solutions to~\eqref{FD} as $t\to\infty$, but we will also prove that, exactly as for the non-weighted equation corresponding to $(\beta,\gamma)=(0,0)$, there is a basin of attraction of $u_\star$ for any $m\in(0,1)$. However, there are many differences with respect to the non-weighted case, which will be summarized in Section~\ref{Sec:Conclusion}. To study the convergence of $u$ to $u_\star$, it is simpler to use \emph{self-similar variables} (see Section~\ref{sect: nonFP} for details) and consider the Fokker-Planck-type equation
\be{FD-FP}
v_t+|x|^\gamma\,\nabla\cdot\left[\,|x|^{-\beta}\,v\,\nabla\big(v^{m-1}-|x|^{2+\beta-\gamma}\big)\right]=0
\ee
with initial condition $v(t=0,\cdot)=v_0$. Self-similar solutions are transformed into Barenblatt-type stationary solutions given by
\[\label{Barenblatt-intro}
\mathfrak B(x):=\(C(M)+|x|^{2+\beta-\gamma}\)^\frac1{m-1}\,,
\]
where $C(M)$ is a positive constant uniquely determined by the \emph{weighted mass} condition
\[
\int_{\R^d}\mathfrak B\,\frac{\rd x}{|x|^\gamma}=M:=\int_{\R^d} v_0\,\frac{\rd x}{|x|^\gamma}\,,
\]
at least if $m\in(m_c,1)$. Altogether, what we aim at is establishing an \emph{exponential} convergence of $v$ to $\mathfrak B$ as $t\to\infty$, when the corresponding distance is measured in terms of the \emph{free energy} (which is sometimes called \emph{generalized relative entropy} in the literature)
\be{eq: Entropy-intro}
\mathcal F[v]:=\frac1{m-1}\int_{\R^d}\left[ v^m-\mathfrak B^m-m\,\mathfrak B^{m-1}\,(v-\mathfrak B) \right] \frac{\rd x}{|x|^\gamma}\,.
\ee
By evolving such free energy along the flow and differentiating with respect to $t$, we formally obtain that
\be{Eqn:DerivF}
\frac \rd{\rd t}\,\mathcal F[v(t)]=-\,\frac m{1-m}\,\mathcal I[v(t)]\,,
\ee
where $\mathcal I[v]$ denotes the \emph{relative Fisher information}
\[\label{eq: Fisher-intro}
\mathcal I[v]:=\int_{\R^d}v\left|\,\nabla v^{m-1}-\nabla\mathfrak B^{m-1}\right|^2\,\frac{\rd x}{|x|^\beta}\,.
\]
This will be proved rigorously in Section~\ref{sect: rig-ent}. Note that, with some abuse of notation, when we write $v(t)$ we mean the whole spatial profile of the function evaluated at time $t$.

\medskip Our goal is to relate $\mathcal F[v]$ and $\mathcal I[v]$, at least as $t\to\infty$, and for that we need a detour by a family of Caffarelli-Kohn-Nirenberg inequalities which have been introduced in~\cite{Caffarelli-Kohn-Nirenberg-84} and studied in Part I of this work,~\cite{BDMN2016a}. Let us explain this a bit more in detail. For all $q \ge 1$, consider the weighted norms
\[
\nrm w{q,\gamma}:=\(\int_{\RR^d}|w|^q\,|x|^{-\gamma}\,\rd x \)^{\frac1q}\quad\mbox{with}\quad\nrm wq:=\nrm w{q,0}\,,
\]
and define $\mathrm L^{q,\gamma}(\R^d)$ as the space of all measurable functions $w$ such that $\nrm w{q,\gamma}$ is finite. Actually, at some points below, we shall also make use of the above definition for $ q \in (0,1) $ (in which cases clearly $\nrm w{q,\gamma}$ is no more a norm). The \emph{Caffarelli-Kohn-Nirenberg interpolation inequalities}
\be{CKN}
\nrm w{2p,\gamma}\le\C_{\beta,\gamma,p}\,\nrm{\nabla w}{2,\beta}^\ttheta\,\nrm w{p+1,\gamma}^{1-\ttheta}
\ee
with $\ttheta:=\frac{(d-\gamma)\,(p-1)}{p\,[d+\beta+2-2\,\gamma-p\,(d-\beta-2)]}$ and parameters $\beta$, $\gamma$, $p$ subject to
\be{parameters-1}
d\ge2\,, \quad \gamma\in(-\infty,d)\,, \quad\gamma-2<\beta<\tfrac{d-2}d\,\gamma
\ee
and
\[\label{parameters-2}
p\in\(1,p_\star\right]\quad\mbox{with}\quad p_\star:=\tfrac{d-\gamma}{d-\beta-2}
\]
are valid for all functions in the space obtained by completion of $\mathcal{D}(\RR^d)$ with respect to the norm $\|\cdot\|$ defined by $w\mapsto\|w\|^2=\nrm{\nabla w}{2,\beta}^2+\nrm w{p+1,\gamma}^2$: see~\cite[Section~2.1]{BDMN2016a} for more details. We shall take for granted once for all assumptions~\eqref{parameters-1} on the parameters, even when it is not mentioned explicitly. However, the sub\-criticality condition $p\le p_\star$ will be assumed for global decay estimates, but not in the study of asymptotic decay estimates.

In~\eqref{CKN}, $\C_{\beta,\gamma,p}$ is meant to be the \emph{best constant}. In Part I of this study,~\cite{BDMN2016a}, we have showed that the equality case is achieved by $w_\star=\mathfrak B^{m-1/2}$ when $p=1/(2\,m-1)$ if \emph{symmetry} holds in~\eqref{CKN}, that is, if optimality is achieved among radial functions. However, we have proved in~\cite{BDMN2016a} that \emph{symmetry breaking} takes place if
\be{def-BFS}
\gamma<0\;\mbox{and}\;\beta_{\rm FS}(\gamma)<\beta<\tfrac{d-2}d\,\gamma\,,\;\mbox{with}\;\beta_{\rm FS}(\gamma):=d-2-\sqrt{(\gamma-d)^2-4\,(d-1)}\,.
\ee
In this case we have
\[
\C_{\beta,\gamma,p}^\star:=\nrm{\nabla w_\star}{2,\beta}^{-\ttheta}\,\nrm{w_\star}{p+1,\gamma}^{\ttheta-1}\,\nrm{w_\star}{2p,\gamma}<\C_{\beta,\gamma,p}\,.
\]

On the contrary, according to~\cite{DELM2015}, symmetry holds, so that $\C_{\beta,\gamma,p}=\C_{\beta,\gamma,p}^\star$, if
\[
0\le\gamma<d\,,\quad\mbox{or}\quad\gamma<0\quad\mbox{and}\quad\beta\le\beta_{\rm FS}\,.
\]
In this case, inequality~\eqref{CKN} can be written as an \emph{entropy -- entropy production} inequality
\[\label{Ineq:E-EP}
(2+\beta-\gamma)^2\,\mathcal F[v]\le\frac m{1-m}\,\mathcal I[v]
\]
and, as a straightforward consequence (recall~\eqref{Eqn:DerivF}), we formally deduce the convergence rate
\be{EntropyDecay}
\mathcal F[v(t)]\le\mathcal F[v(0)]\,e^{-\,2\,(1-m)\,\Lambda_\star\,t}\quad\forall\,t\ge0\quad\mbox{with}\quad\Lambda_\star:=\tfrac{(2+\beta-\gamma)^2}{2\,(1-m)}\,.
\ee
This connection is well known and goes back to~\cite{MR1940370}, where similar properties were investigated in the non-weighted case, namely for $(\beta,\gamma)=(0,0)$. For more details see~\cite[Proposition~1]{BDMN2016a}. Conditions have to be given to make Estimate~\eqref{EntropyDecay} rigorous. In fact much more is known. Let us consider again the \emph{entropy -- entropy production} inequality
\be{EP}
\mathcal K(M)\,\mathcal F[v]\le\mathcal I[v]\quad\forall\,v\in\L^{1,\gamma}(\R^d)\quad\mbox{such that}\quad \nrm v{1,\gamma}=M\,,
\ee
where $\mathcal K(M)$ is the best constant, which may possibly take the value $0$. With $\Lambda(M):=\frac m2\,(1-m)^{-2}\,\mathcal K(M)$, we formally deduce from~\eqref{Eqn:DerivF} and~\eqref{EP} that
\be{GlobalEntropyDecay}
\mathcal F[v(t)]\le\mathcal F[v(0)]\,e^{-\,2\,(1-m)\,\Lambda(M)\,t}\quad\forall\,t\ge0\,.
\ee
For brevity, this is what we shall call a \emph{global rate of decay} because the dependence in $v$ at time $t=0$ is explicitly given by $\mathcal F[v(0)]$. Let
\[
m_1:=\frac{2\,d-2-\beta-\gamma}{2\,(d-\gamma)}
\]
and denote by $\Lambda_{0,1}$ the lowest eigenvalue associated with non-radial eigenfuntions of the operator $\mathcal L$ defined by
\[
\mathcal L\,f:=|x|^\gamma\,\mathfrak B^{m-2}\,\nabla\cdot\(\,|x|^{-\beta}\,\mathfrak B\,\nabla f\)\,,
\]
where $\mathfrak B$ denotes the Barenblatt profile defined by
\be{Barenblatt-intro2}
\mathfrak B(x)=\(C+|x|^{2+\beta-\gamma}\)^\frac1{m-1}\quad\forall\,x\in\RR^d
\ee
for some $C>0$. Of course we shall later take $C=C(M)$ if $m\in(m_c,1)$. Based on variational methods, the following result has been proved in~\cite{BDMN2016a}. We shall also give a short additional proof of (i) below in Section~\ref{Sec:EPOpbFS}.
\begin{Thm}\label{Thm:BDMN-I} Let~\eqref{parameters-1} hold and $m\in[m_1,1)$. With the above notations, we have:
\begin{enumerate}
\item[(i)] For any $M>0$, if $\Lambda(M)=\Lambda_\star$ then $\beta=\beta_{\rm FS}(\gamma)$,
\item[(ii)] If $\beta>\beta_{\rm FS}(\gamma)$ then $\Lambda_{0,1}<\Lambda_\star$ and $\Lambda(M)\in(0,\Lambda_{0,1}]$ for any $M>0$,
\item[(iii)] For any $M>0$, if $\beta<\beta_{\rm FS}(\gamma)$ and $\gamma<0$, or if $0\le\gamma<d$, then $\Lambda(M)>\Lambda_\star$,
\end{enumerate}
where $\beta_{\rm FS}(\gamma) $ is given by~\eqref{def-BFS}.
\end{Thm}
We shall assume that the initial datum~$v_0$ is sandwiched between two Barenblatt profiles: \emph{there exist two positive constants $C_1$ and $C_2$ such that}
\be{Ineq:sandwiched}
\mathfrak B_1(x):=\(C_1+|x|^{2+\beta-\gamma}\)^\frac1{m-1} \le v_0(x)\le\(C_2+|x|^{2+\beta-\gamma}\)^\frac1{m-1}=:\mathfrak B_2(x) \ \ \forall\,x\in\R^d\,.
\ee
Condition~\eqref{Ineq:sandwiched} may look rather restrictive, but it is probably not, because it is expected that the condition is satisfied, for some positive $t$, by any solution with nonnegative initial datum having finite initial free energy, as it is the case when $(\beta,\gamma)=(0,0)$ and for $m$ sufficiently close to~$1$: see for instance~\cite{BV-JFA,BBDGV}. However, initial regularization effects of~\eqref{FD-FP} are out of the scope of the present paper. If $m$ is not  close to~$1$, Condition~\eqref{Ineq:sandwiched} determines, to some extent, the basin of attraction of $\mathfrak B$. What we shall prove is the following result on \emph{global rates}.
\begin{Cor}\label{Cor:BDMN-I} Let~\eqref{parameters-1} hold and $m\in[m_1,1)$. With the above notations,~\eqref{GlobalEntropyDecay} holds if $v$ is a solution to~\eqref{FD-FP} with initial datum subject to~\eqref{Ineq:sandwiched}.
\end{Cor}

We know that symmetry holds in~\eqref{CKN} whenever $\beta\le\beta_{\rm FS}(\gamma)$ and $\gamma<0$, or if $0\le\gamma<d$. The restriction $m\ge m_1$ comes from the (sub-)criticality condition $p\le p_\star$ in~\eqref{CKN}. As in~\cite{BBDGV-CRAS,BBDGV,BDGV}, if one is interested only in the \emph{asymptotic rate of decay} of $\mathcal F[v(t)]$ as $t\to\infty$ (\emph{i.e.},~without requiring that the multiplicative constant is $\mathcal F[v(0)] $), this restriction can be lifted and better estimates of the rates can be given using an appropriate linearization of the problem. Let us give some explanations in this regard.

\smallskip Still at a formal level, we may consider a solution $v=\mathfrak B\,(1+\varepsilon\,\mathfrak B^{1-m}\,f)$ to~\eqref{FD-FP} and keep only the first order term in $\varepsilon$, as in~\cite{BBDGV}. The corresponding linear evolution equation is
\[
f_t=(1-m)\,\mathcal L\,f\,.
\]
Mass conservation (more precisely, \emph{relative} mass conservation, see Sections~\ref{Sec:Basic} and~\ref{sect: nonFP}), which is taken into account by requesting that
\[
\int_{\R^d} f\,\mathfrak B^{2-m}\,\frac{\rd x}{|x|^\gamma}=0\,,
\]
suggests to analyse the spectral gap of $\mathcal L$ considered as an operator acting on the space $\mathrm L^2(\R^d,\mathfrak B^{2-m}\,|x|^{-\gamma}\,\rd x)$. The reader interested in more details is invited to refer to~\cite[Section~3.3]{BDMN2016a}. Let us define
\be{parameters}
m_\ast:=\frac{d-4-2\,\beta+\gamma}{d-2-\beta}\,,\quad\alpha:=1+\frac{\beta-\gamma}2\,,\quad\delta:=\frac1{1-m}\,,\quad n:=2\,\frac{d-\gamma}{\beta+2-\gamma}\,,
\ee
and pick the unique positive solution to $\nueta\,(\nueta+n-2)=(d-1)/\alpha^2$, which is given by
\[
\nueta=\sqrt{\tfrac{d-1}{\alpha^2}+\big(\tfrac{n-2}2\big)^2}-\tfrac{n-2}2=\tfrac2{2+\beta-\gamma}\sqrt{d-1+\big(\tfrac{d-2-\beta}2\big)^2}-\tfrac{d-2-\beta}{2+\beta-\gamma}\,.
\]
The following result can be deduced from~\cite[Lemma~8]{BDMN2016a} (please note the discrepancy of a factor $\alpha^2$, which is due to the change of variables $x\mapsto x\,|x|^{\alpha-1}$). See in particular~\cite[Figure~4 and Appendix~B]{BDMN2016a} for details.
\begin{Prop}[A Hardy-Poincar\'e-type inequality]\label{Prop:Spectrum} Let $d\ge2$, $m\in(0,1)$ and~\eqref{parameters-1} holds. For any $f\in\mathrm L^2(\R^d,\mathfrak B^{2-m}\,|x|^{-\gamma}\,\rd x)$, such that $\int_{\R^d}f\,\mathfrak B^{2-m}\,|x|^{-\gamma}\,\rd x=0$ if $m>m_\ast$, there holds
\[
\int_{\R^d}|\nabla f|^2\,\mathfrak B\,\frac{\rd x}{|x|^\beta}\ge\Lambda\int_{\R^d}|f|^2\,\mathfrak B^{2-m}\,\frac{\rd x}{|x|^\gamma}\,.
\]
The optimal constant is $\Lambda=\Lambda_{\rm ess}$ if $\delta\le(n+2)/2$ and $\Lambda=\min\left\{\Lambda_{\rm ess},\Lambda_{0,1},\Lambda_{1,0}\right\}$ otherwise, with
\[
\Lambda_{\rm ess}=\tfrac14\,\(n-2-2\,\delta\)^2\,,\quad\Lambda_{0,1}=\,2\,\delta\,\nueta\,,\quad\Lambda_{1,0}=2\,(2\,\delta-n)\,.
\]
\end{Prop}
Notice that the optimal constant $\Lambda$ in the Hardy-Poincar\'e-type inequality is independent of $C$. In the range $m\in(m_c,1)$, it is also independent of $M$ since~$\mathfrak B$ as defined in~\eqref{Barenblatt-intro2} with $C=C(M)$ explicitly depends on $M$, according to the expression of $C(M)$ given in~\cite[Appendix~A]{BDMN2016a}.

In Proposition~\ref{Prop:Spectrum}, $\Lambda_{\rm ess}$, $\Lambda_{0,1}$ and $\Lambda_{1,0}$, respectively, denote the infimum of the essential spectrum of $\mathcal L$, the lowest positive eigenvalue associated with a non-radial eigenfunction, and the lowest positive eigenvalue associated with a radial eigenfunction. In practice, we have $\Lambda=\Lambda_{1,0}$ if $(n+2)/2\le\delta\le n/(2-\nueta)$ and $\Lambda=\Lambda_{0,1}$ if $\delta > n/(2-\nueta)$, with $(n-2)/2<(n+2)/2<n/(2-\nueta)$ and $\Lambda_{\rm ess}=0$ if and only if $\delta=(n-2)/2$. See Figure~\ref{Fig1}.
\begin{figure}[ht]
\hspace*{-6pt}\includegraphics[width=6.25cm]{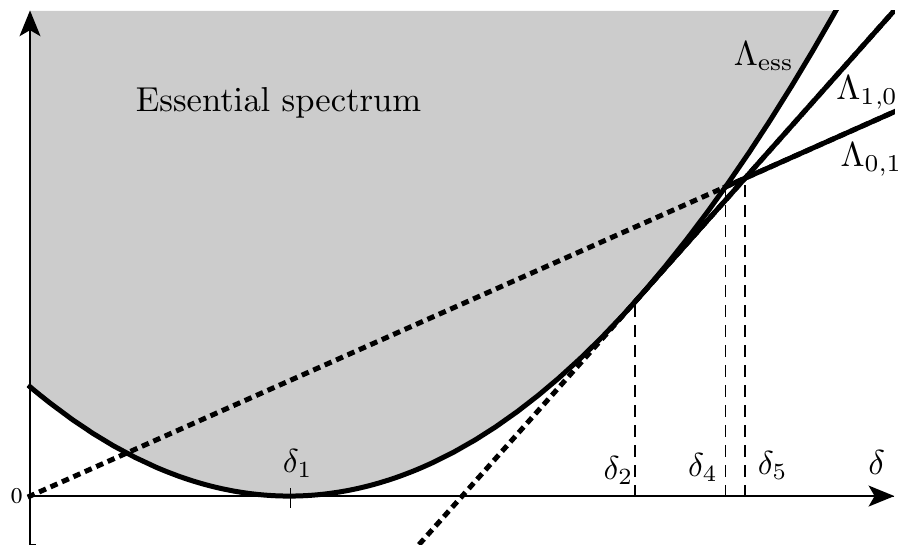}\hspace*{6pt}\includegraphics[width=6cm]{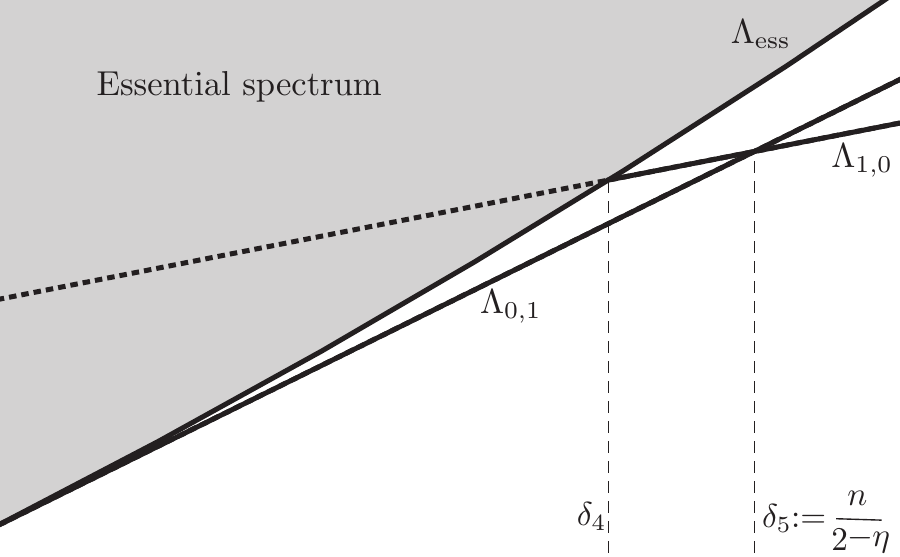}
\caption{\label{Fig1} The spectrum of $\mathcal L$ as a function of $\delta=\frac1{1-m}$, with $n=5$. The essential spectrum corresponds to the grey area, and its bottom is determined by the parabola $\delta\mapsto\Lambda_{\rm ess}(\delta)$. The two eigenvalues $\Lambda_{0,1}$ and $\Lambda_{1,0}$ are given by the plain, half-lines, away from the essential spectrum. Note that solutions of the eigenvalue problem exist for any value of $\delta$ but may not be in the domain of the operator or below the essential spectrum and are then represented as dotted half-lines. See \cite[Appendix~B]{BDMN2016a} for a discussion of the values of $\delta_1$, $\delta_2$,... $\delta_5$. The right figure is an enlargement of the left~one. This configuration is not generic: see \cite[Fig.~5]{BDMN2016a} for other cases.}
\end{figure}

\smallskip Notice that $m_\ast$ as defined in~\eqref{parameters} is the unique value of $m$ for which, eventually, $\Lambda_{\rm ess}=0$ and, as a consequence, for which there is no spectral gap. Notice that for some values of $d$, $\gamma$ and $\beta$, the exponent $m_\ast$ takes nonpositive values. However our results are limited to $m\in(0,1)$ and in particular $m>0$ will be assumed throughout this paper. Since $\delta=1/(1-m)$, we obtain precisely the value given by~\eqref{parameters}. If $m>m_\ast$ then $\mathfrak B_2-\mathfrak B_1$ is in $\mathrm L^{1,\gamma}(\R^d)$, where $\mathfrak B_1$ and $\mathfrak B_2$ are defined as in~\eqref{Ineq:sandwiched}. This is not anymore true if $m\le m_\ast$. In that case we shall consider $\LL^{1,\gamma}$-perturbations of $\mathfrak B$ as defined in~\eqref{Barenblatt-intro2}, for some constant $C\in[C_2,C_1]$. If $m>m_\ast$ then the condition
\[\label{cond-zero-mass}
\int_{\RR^d}{(v_0-\mathfrak B)\,|x|^{-\gamma}\,\rd x }=0
\]
uniquely determines $C=C(M)$. We shall refer to these conditions as the \emph{relative mass condition}: see Assumptions (H1) and (H2) in Section~\ref{Sec:Asymptotic} for further details.

\smallskip We are now able to state the main results of this paper.
\begin{Thm}\label{Thm:Asymptotic rates} Let~\eqref{parameters-1} hold and $m\in(0,1)$, with $m\neq m_\ast$. Under the relative mass condition and with same notations as in Proposition~\ref{Prop:Spectrum}, if $v$ solves~\eqref{FD-FP} subject to~\eqref{Ineq:sandwiched}, then there exists a positive constant~$\mathcal C$ such that
\[\label{AsymptoticEntropyDecay}
\mathcal F[v(t)]\le\mathcal C\,e^{-\,2\,(1-m)\,\Lambda\,t}\quad\forall\,t\ge0\,.
\]
\end{Thm}
As in~\cite{BBDGV,BDGV}, this result itself relies on a result of \emph{relative uniform convergence} which is the key estimate to relate the free energy to the spectrum of $\mathcal L$. Before stating the latter, let us define
\[
\textstyle\zeta :=1 -\,\big(1-\frac{2-m}{(1-m)\,q}\big)\,\big(1-\frac{2-m}{1-m}\,\varth\big)\quad\mbox{where}\quad\varth:=\frac{(1-m)\,(2+\beta-\gamma)}{(1-m)\,(2+\beta)+2+\beta-\gamma}
\]
and observe that in view of hypotheses~\eqref{parameters-1} and of the change of variables~\eqref{parameters}, $\varth$ is in the range $0<\varth < \tfrac{1-m}{2-m} < 1$.
\begin{Thm}\label{Thm:RUC} Under the assumptions of Theorem~\ref{Thm:Asymptotic rates}, there exist positive constants~$\mathcal K$ and~$t_0$ such that, for all $q\in\big[\tfrac{2-m}{1-m},\infty\big]$, the function $w=v/\mathfrak B$ satisfies
\be{AsymptoticDecay.RelErr}
\left\| {w(t)-1} \right\|_{\LL^{q,\gamma}(\RR^d)} \le \mathcal K\,e^{-\,2\,\frac{(1-m)^2}{2-m}\,\Lambda\,\zeta\,(t-t_0)}\quad\forall\,t\ge t_0
\ee
in the case $\gamma \in (0,d)$, and
\[\label{AsymptoticDecay.RelErr-bis}
\left\| {w(t)-1} \right\|_{\LL^{q,\gamma}(\RR^d)} \le \mathcal K\,e^{-\,2\,\frac{(1-m)^2}{2-m}\,\Lambda\,(t-t_0)}\quad\forall\,t\ge t_0
\]
in the case $\gamma \le 0$.
\end{Thm}
We point out that Estimate~\eqref{AsymptoticDecay.RelErr} yields an improvement of a similar result, namely~\cite[Theorem~3]{BBDGV}, in the non-weighted case $(\beta,\gamma)=(0,0)$. We shall comment more on the rates of convergence provided by Theorem~\ref{Thm:RUC} in Section~\ref{Sec:Conclusion}.

\smallskip The proof of Theorem~\ref{Thm:RUC} partially relies on uniform H\"older-regularity estimates for bounded solutions to a linearized version of Equation~\eqref{FD-FP}. In view of possible degeneracies or singularities of the weights $ |x|^{-\gamma} $ and $ |x|^{-\beta} $ at the origin, such results do not follow from standard parabolic theory and therefore have to be proved separately. We devote an Appendix to these issues, where we give sketches of proofs. These are based on a strategy developed for similar equations by~\cite{ChSe85}.

\smallskip Using refinements that will be discussed in Section~\ref{Sec:Refinements}, we can also prove convergence results in $\LL^{1,\gamma}$ norms. For this purpose, we need to restrict the range of $m$ to $(\widetilde m_1,1)$, where $\widetilde m_1$ is the smallest number such that $\int_{\RR^d} |x|^{2+\beta-\gamma}\,\mathfrak B\,|x|^{-\gamma}\,\rd x $ is finite for all $m \in (\widetilde m_1,1)$, that is
\[
\widetilde m_1:=\frac{d-\gamma}{d+2+\beta-2\,\gamma}\,.
\]
Let us introduce the rescaled function
\[
v_\mu(t,x):=\mu^{\gamma-d}\,v(t,x/\mu)\quad\forall\,(t,x)\in\R^+\times\R^d\,, \ \forall\,\mu >0\,.
\]
\begin{Thm}\label{Thm:Norms} Under the assumptions of Theorem~\ref{Thm:Asymptotic rates} with in addition $m\in(\widetilde m_1,1)$, there exists a monotone, positive function $t\mapsto\mu(t)$ with $\lim_{t\to+\infty}\mu(t)=1$ such that
\[
\limsup_{t\to\infty} e^{(1-m)\,\min\{\Lambda_{\rm ess},\Lambda_{0,1}\}\,t}\,\nrm{\(v_{\mu(t)}(t)-\mathfrak B\)\(1+|x|^{2+\beta-\gamma}\)}{1,\gamma}<\infty\,.
\]
\end{Thm}
This result is an improvement in the spirit of~\cite{1004} in the non-weighted case $(\beta,\gamma)=(0,0)$. As it appears in Proposition~\ref{Prop:Spectrum} and Theorems~\ref{Thm:Asymptotic rates} and~\ref{Thm:RUC}, $\Lambda$ is smaller than $\min\{\Lambda_{\rm ess},\Lambda_{0,1}\}$ only under conditions that are discussed in~\cite[Appendix~B]{BDMN2016a}. Also notice that, by undoing the self-similar change of variables outlined in Section~\ref{sect: nonFP}, it is possible to give algebraic rates of convergence for the original solutions to~\eqref{FD}, as in~\cite{BBDGV}.

\medskip Let us conclude this introduction by a few bibliographical references. In the case without weights, we primarily refer to~\cite{BBDGV-CRAS,BBDGV,BDGV} and references therein. The special case corresponding to $\delta=(n-2)/2$ has been treated in~\cite{BGVm*}. Still in the non-weighted case, improvements have been obtained more recently in~\cite{1004,DT2011,MR3103175,1751-8121-48-6-065206,1751-8121-48-6-065206,1501,dolbeault:hal-01081098} using refinements of relative entropy methods, and in~\cite{DKM} using a detailed analysis of fast and slow variables and of the invariant manifolds. These papers are anyway limited to the choice $(\beta,\gamma)=(0,0)$. More references can be found therein. 

As for problems with power law weights, we shall refer to~\cite{MR2425009,MR2461559,MR2556497} for approaches based on comparison techniques in the case $\beta=0$ and for the porous media equation. The papers \cite{NR13,IS14} deal with the critical power $ |x|^{-2} $, where asymptotics is more subtle. A detailed long-time analysis has been carried out in~\cite{GMP15} for the \emph{fractional} porous media equation with a weight.  Diffusion equations of porous media type with \emph{two} weights (\emph{i.e.}~weights having the same role as $ |x|^{-\gamma} $ and $ |x|^{-\beta} $ here) have been investigated, \emph{e.g.},~in~\cite{DGGW8,GMPo13}, where well-posedness issues as well as smoothing effects and asymptotic estimates are discussed in rather general weighted frameworks by means of functional inequalities. In the fast diffusion regime, convergence in relative error to a separable profile for radial solutions on the \emph{hyperbolic space} has been proved by \cite{GM14}, through pure barrier methods. Note that, in radial coordinates, the Laplace-Beltrami operator is in fact a two-weight Laplacian. The corresponding analysis for the porous medium equation (for general solutions) has then been carried out in \cite{JLV14}.

A detailed justification of the introduction of weights and especially power law weights in case of porous media and fast diffusion equations can be found in~\cite{MR634287,MR637497}. In~\cite{DMN2015}, for $\beta=0$ and $\gamma>0$ small enough, symmetry of optimal functions in the Caffarelli-Kohn-Nirenberg inequalities~\eqref{CKN} is proved to hold. Notice however that the inequalities are then more of Hardy-Sobolev type than of Caffarelli-Kohn-Nirenberg because only one weight is involved. The other case of symmetry in Caffarelli-Kohn-Nirenberg inequalities, which is now fully understood, is the one corresponding to the threshold case $p=p_\star$, which has been recently solved in~\cite{DEL2015}. Remarkably the proof relies on the very same flow~\eqref{FD} and an approach based on the Bakry-Emery $\Gamma_2$ method. The reader interested in further considerations on Caffarelli-Kohn-Nirenberg inequalities, $\Gamma_2$ computations and rigidity results in nonlinear elliptic problems on compact and non-compact manifolds is invited to refer to this paper for a more complete review of the literature in this direction.  

The paper is organized as follows. Properties of self-similar solutions, an existence result, a comparison result, the conservation of the relative mass, results on the relative entropy and the rewriting of \eqref{FD} in relative variables after a self-similar change of variables have been collected in Section~\ref{Sec:Prelim}. These results are adapted form the case $(\beta,\gamma)=(0,0)$. Section~\ref{Sec:Main} is devoted to regularity issues and to the \emph{relative uniform convergence}, that is, the uniform convergence of the quotient of the solution in self-similar variables by the Barenblatt profile: this at the core of our results and it is also where our paper differs from the case $(\beta,\gamma)=(0,0)$. There we prove Theorems~\ref{Thm:Asymptotic rates} and~\ref{Thm:RUC}. Because of the weights, the H\"older regularity at the origin is an issue. It relies on a technical result, based on an adaptation of \cite{ChSe85}: the proof is given in an Appendix. Some additional results, including the proof of Theorem~\ref{Thm:Norms} and some comments have been collected in Section~\ref{Sec:Additional}.

Throughout this paper, $B_\rho$ denotes the centered ball of radius $\rho$, that is, $B_\rho:=\{x\in\RR^d\,:\,|x|<\rho\}$.

\section{Self-similar variables, relative entropy and large time asymptotics}\label{Sec:Prelim}

\subsection{The self-similar solutions}\label{Sec:Asymptotic}

In order to avoid confusion between original variables and rescaled variables, let us rewrite~\eqref{FD} as
\[\label{FDE.problem}
|y|^{-\gamma}\,u_\tau+\nabla \cdot \left( |y|^{-\beta}\,u\,\nabla{u^{m-1}} \right)=0 \quad \forall\,(\tau,y) \in \RR^+ \times \RR^d\,.
\]
The \emph{whole} family of explicit self-similar solutions of Barenblatt type is given by
\be{Barenblatt.FPvars1}
U_{C,T}(\tau,y):=\frac{R(\tau)^{\gamma-d}}{\(C + \big(|y|/R(\tau)\big)^{2+\beta-\gamma}\)^{1/(1-m)}} \quad \forall\,(\tau,y) \in \RR^+ \times \RR^d\,,
\ee
where $C$, $T>0$ are free parameters and $ R(\tau) $ is defined by
\be{eq: R-tau}
\frac{\rd R}{\rd \tau}=R^{1-(d-\gamma)\,(m-m_c)}\,, \quad R(0)=\left[(d-\gamma)\,|m-m_c|\,T \right]^{\frac1{(d-\gamma)\,(m-m_c)}}
\ee
if $m\neq m_c$, with $m_c$ as above, namely
\[\label{eq-def-mc0}
m_c:=\frac{d-2-\beta}{d-\gamma} \in (0,1)\,.
\]
In the special case $m=m_c$ we shall replace the initial condition with $ R(0)=e^{T}$, for $ T \in \mathbb{R}$. More explicitly,
\[
R(\tau)=
\begin{cases}
\left[(d-\gamma)\,(m-m_c)\,(T+\tau) \right]^{\frac1{(d-\gamma)\,(m-m_c)}} & \textrm{if}\;m \in (m_c,1)\,,\;\forall\,\tau \ge 0\,, \\[6pt]
e^{T+\tau} & \textrm{if}\; m=m_c\,,\;\forall\,\tau \ge 0\,, \\[6pt]
\left[(d-\gamma)\,(m_c-m)\,(T-\tau) \right]^{-\frac1{(d-\gamma)\,(m_c-m)}} & \textrm{if}\;m \in (0,m_c)\,,\;\forall\,\tau \in [0,T)\,.
\end{cases}
\]
If $ m \ge m_c $ Barenblatt-type solutions are positive for all $\tau>0$. If $ m<m_c $ these solutions extinguish at $\tau=T$.

\medskip As already mentioned in Section~\ref{Sec:Intro}, we shall require that the initial datum \hbox{$u(0)=u_0$} is trapped between two Barenblatt profiles. More precisely:

\medskip\noindent {\bf (H1)} There exist positive constants $T$ and $C_1>C_2 $ such that
\[
U_{C_1,T}(0,y)\le u_0(y)\le U_{C_2,T}(0,y) \quad \forall\,y \in \RR^d\,.
\]
\noindent {\bf (H2)} There exist $ C \in [C_2,C_1]$ and $f\in\LL^{1,\gamma}(\RR^d)$ such that
\[
u_0(y)=U_{C,T}(0,y)+f(y) \quad \forall\,y \in \RR^d\,.
\]

If $m < m_c$ solutions with initial datum as above extinguish at $t=T<\infty$ as we shall deduce from the comparison principle (see Corollary~\ref{Lem:MP} and related comments below). Such solutions do do not belong to $\LL^{1,\gamma}(\RR^d)$. On the other hand, if $ m \ge m_c $ solutions are positive at all $\tau>0$. They belong to $\LL^{1,\gamma}(\RR^d) $ if in addition $m>m_c$. If $m>m_\ast$, with
\[
m_\ast:=\frac{d-4-2\,\beta+\gamma}{d-2-\beta} < m_c\,,
\]
Assumption (H2) is in fact a consequence of (H1). Indeed in such a range Barenblatt solutions may not be in $\LL^{1,\gamma}(\RR^d) $ but the difference of two Barenblatt profiles still belongs to $\LL^{1,\gamma}(\RR^d)$. On the contrary, if $m\le m_\ast$ then (H2) induces an additional restriction.

\subsection{Existence, comparison and conservation of relative mass}\label{Sec:Basic}

In agreement with~\cite{HP}, we provide the following definition of a weak solution.
\begin{Def}\label{def-very-weak}
For a nonnegative $ u_0 \in \L^\infty(\RR^d)$, by a solution to~\eqref{FD} we mean a nonnegative function $ u \in C([0,\infty);\LL^{1,\gamma}_{\rm loc}(\RR^d)) \cap \LL^\infty(\RR^+ \times \RR^d) $ satisfying
\[\label{weak-sol-FDE}
-\int_{\RR^+} \int_{\RR^d} u\,\varphi_\tau\,|y|^{-\gamma} \rd y\,\rd \tau=\frac{1-m}{m}\,\int_{\RR^+} \int_{\RR^d} u^{m}\,\nabla \cdot \left( |y|^{-\beta}\,\nabla{\varphi} \right) \rd y\,\rd \tau
\]
for all $\varphi \in \mathcal{D}(\RR^+ \times \RR^d ) $ and $\lim_{\tau \to 0^+} u(\tau)=u_0$.\end{Def}
In~\cite{HP}, $(\beta,\gamma)=(0,0)$ and we point out that it is only required that $ u_0 \in \LL^1_{\rm loc}(\RR^d) $ and $u \in C([0,\infty);\LL^1_{\rm loc}(\RR^d))$. However, because of the weight $ |y|^{-\beta}$, \emph{a priori} the equation may not make sense, since in general $ u^m \not \in \LL^{1,\beta}_{\rm loc}(\RR^d)$. Hence, for simplicity, we also assume that initial data and solutions are globally bounded, as in the sequel we shall only deal with this kind of solutions.
\begin{Prop}[Existence] \label{prop: existence} Assume that $m\in(0,1)$. For any nonnegative $ u_0 \in \L^\infty(\RR^d) $ there exists a solution to~\eqref{FD} in the sense of the above definition.
\end{Prop}
\begin{proof} We refer the reader to the proof of~\cite[Theorem~2.1]{HP}: minor changes have to be implemented in order to adapt it to our weighted context. The basic idea consists in approximating the initial datum, \emph{e.g.},~with the sequence $ u_{0n}:=\phi_n\,u_0 \in \LL^{1,\gamma}(\RR^d) \cap \LL^\infty(\RR^d)$ where $\phi_n=\phi(\cdot/n)$ and $\phi$ is a smooth truncation function such that $0\le\phi\le1$, $\phi(x)=0$ if $|x|>2$ and $\phi(x)=1$ if $|x|\le1$. The corresponding sequence of solutions $ u_n $ is well defined in view of standard $\LL^{1,\gamma} $ theory (there is no additional difficulty due to the weights compared to the standard theory as exposed in~\cite{VazBook}). One can then pass to the limit on such a sequence by exploiting local $\LL^{1,\gamma} $ estimates (as in~\cite[Lemma~3.1]{HP}) along with the global bound $\| u_n(\tau) \|_{\infty} \le \| u_{0n} \|_\infty$, valid for all $\tau >0$.
\end{proof}
\begin{Prop}[$\LL^{1,\gamma}$-contraction]\label{Prop:contraction} Assume that $m\in(0,1)$ and let $ u_{01}$, $u_{02} \in \LL^\infty(\RR^d) $ be any two nonnegative initial data with corresponding solutions $ u_1$,~$u_2$ to~\eqref{FD}, that are  constructed via the approximation scheme of the proof of Proposition~\ref{prop: existence}. Then
\[\label{eq:contraction}
\int_{\RR^d} \left[ u_1(\widetilde{\tau},y)-u_2(\widetilde{\tau},y) \right]_+ \frac{\dy}{|y|^\gamma} \le \int_{\RR^d} \left[ u_1(\tau,y)-u_2(\tau,y) \right]_+ \frac{\dy}{|y|^\gamma} \quad \forall\,\widetilde{\tau} \ge \tau \ge 0\,.
\]
\end{Prop}
\begin{proof} The inequality holds for the approximate solutions $ u_{1,n}$ and $ u_{2,n} $ still as a consequence of the standard $\LL^{1,\gamma} $ theory. Hence, the assertion just follows by taking limits as $ n \to \infty$.
\end{proof}
Proposition~\ref{Prop:contraction} trivially implies the following key comparison result.
\begin{Cor}[Comparison principle]\label{Lem:MP}
Under the same hypotheses as in Proposition~\ref{Prop:contraction}, if $ u_{01} \le u_{02} $, then $ u_1(\tau) \le u_2(\tau) $ for all $\tau \ge 0$.
\end{Cor}
As the reader may note, we do not claim that we have a comparison principle (and hence a uniqueness result) for \emph{any} solutions in the sense of Definition~\ref{def-very-weak}, but only for those obtained as limits of $\LL^{1,\gamma} $ approximations. Nevertheless, in the sequel by \emph{solution} we shall tacitly mean the one constructed as in the proof of Proposition~\ref{prop: existence}, for which comparison holds. Since we consider initial data satisfying (H1), in order to conclude that the corresponding solutions are trapped between Barenblatt profiles at any time, one has to check that the self-similar solution given by~\eqref{Barenblatt.FPvars1} can also be obtained as a limit of $\LL^{1,\gamma} $ approximate solutions. This is a standard fact given the explicit profile of $ U_{C,T}$.

\smallskip Mass conservation is used in the range $m>m_c$ to determine the parameter~$C=C(M)$ which characterizes the Barenblatt profile $U_{C,T}$ having the same mass as $ u$. In the range $m\le m_c$ we can still prove that the quantity
\[
\int_{\RR^d}[u(\tau,y)-U_{C,T}(\tau,y)]\,|y|^{-\gamma} \rd y\,,
\]
which we shall refer to as \emph{relative mass}, is conserved at any $\tau>0$, even if $U_{C,T}(\tau)\not\in \LL^{1,\gamma}(\RR^d)$.
\begin{Prop}[Conservation of relative mass]\label{prop:relconsmass} Assume that $m\in(0,1)$ and consider a solution $u$ of~\eqref{FD} with initial datum $u_0$ satisfying {\rm (H1)-(H2)}. Then
\[\label{integral.eq}
\int_{\RR^d}\left[u(\tau,y)-U_{C,T}(\tau,y)\right]\frac{\dy}{|y|^\gamma}
=\int_{\RR^d}\left[u_0(y)-U_{C,T}(0,y)\right]\frac{\dy}{|y|^\gamma} \quad \forall\,\tau \ge 0\,.
\]
\end{Prop}
\begin{proof} We proceed along the lines of the proof of~\cite[Proposition~1]{BBDGV}. That is, let~$\phi$ be a $C^2(\RR^+)$ function such that $0\le\phi\le1$, $\phi(x)=0$ if $|x|>2$ and $\phi(x)=1$ if $|x|\le1$. For any $\lambda>0$, set $\phi_\lambda(y):=\phi(|y|/\lambda)$. Then
\begin{multline*}
\left|\dfrac{\rd}{\rd\tau}\int_{\RR^d}\left[u(\tau,t)-U_{C,T}(\tau,y)\right] \phi_\lambda\,\frac{\rd y}{|y|^\gamma}\right| \\
=\,\frac{1-m}{m}\left|\int_{B_{2\lambda}\setminus B_\lambda} \left[u^m(\tau,y)-U_{C,T}^m(\tau,y) \right] \nabla \cdot \left(|y|^{-\beta}\,\nabla \phi_\lambda \right) \rd y \right|\hspace*{2cm}\\
\le\,(1-m)\,\int_{B_{2\lambda}\setminus B_\lambda}U_{C_1,T}^{m-1}\,\left|u-U_{C,T}\right|\left( |y|^{-\beta} \left|\Delta\phi_\lambda\right| + \beta\,|y|^{-\beta-1} \left| \nabla \phi_\lambda \right| \right) \rd y\,.
\end{multline*}
As $\lambda\to\infty$, we observe that $U_{C_1,T}^{m-1}$, $|\Delta\phi_\lambda|$ and $ |\nabla \phi_\lambda| $ behave like $\lambda^{2+\beta-\gamma}$, $\lambda^{-2}$ and $\lambda^{-1}$, respectively, in the region $B_{2\lambda}\setminus B_\lambda$. In particular,
\[
U_{C_1,T}^{m-1}\left( |y|^{-\beta} \left|\Delta\phi_\lambda\right| + \beta\,|y|^{-\beta-1} \left| \nabla \phi_\lambda \right| \right) \le c\,|y|^{-\gamma} \quad \forall\,y \in B_{2\lambda}\setminus B_\lambda
\]
for a suitable $c>0$ independent of $\lambda$. Hence, for all $\tau_2> \tau_1 \ge 0$ we deduce that
\begin{multline*}\label{eq2.2-contr}
\left| \int_{\RR^d}\left[u(\tau_2,y)-U_{C,T}(\tau_2,y)\right] \phi_\lambda\,\frac{\rd y}{|y|^\gamma}
-\,\int_{\RR^d}\left[u(\tau_1,y)-U_{C,T}(\tau_1,y)\right] \phi_\lambda\,\frac{\rd y}{|y|^\gamma}\right| \\
\le\,(1-m)\,c\,\int_{\tau_1}^{\tau_2} \int_{B_{2\lambda}\setminus B_\lambda} \left| u(\tau,y)-U_{C,T}(\tau,y) \right| \frac{\rd y}{|y|^\gamma}\,\rd\tau\,.
\end{multline*}
The $\LL^{1,\gamma}$-contraction of Proposition~\ref{Prop:contraction} ensures that the r.h.s.~vanishes as~\hbox{$\lambda\to\infty$}.

\smallskip Actually, since \emph{a priori} $ u_\tau $ does not exist as a function, we have to use a test function $\phi_\lambda$ that depends on time and whose time derivative approximates the difference between two Dirac deltas at times $\tau_2 $ and $\tau_1$. However, this is a standard technicality, which we omitted in order to make the proof more readable.\end{proof}

\subsection{Self-similar variables: a nonlinear Fokker-Plank equation}\label{sect: nonFP}

As already discussed in~\cite{BBDGV,BDMN2016a} and outlined in Section~\ref{Sec:Intro}, it is convenient to analyse the asymptotic behaviour of solutions to~\eqref{FD} whose initial data comply with (H1)-(H2) by means of a suitable time-space change of variables that makes Barenblatt profiles stationary.

Let us rescale the function $u$ according to
\[\label{eq: resc}
u(\tau,y)=R(\tau)^{\gamma-d}\,v\left( \log \frac {R(\tau)} {R(0)},\frac y {R(\tau)} \right), \quad t=\log \frac {R(\tau)} {R(0)}\,, \quad x=\frac y {R(\tau)}\,,
\]
where $R$ is defined by~\eqref{eq: R-tau}. Similarly, the Barenblatt solution $U_{C,T}$ is transformed into $\mathfrak B$ as defined by~\eqref{Barenblatt-intro2}. Straightforward computations show that $v$ solves the nonlinear, weighted Fokker-Plank equation
\be{eq.FDE.FP}
|x|^{-\gamma}\,v_t=-\,\nabla\cdot\Big(|x|^{-\beta}\,v\,\nabla \left(v^{m-1}-\mathfrak B^{m-1} \right)\Big)\quad \forall\,(t,x) \in \mathbb{R}^+ \times \mathbb{R}^d\,.
\ee
In terms of the initial datum $v_0(x)=R(0)^{d-\gamma}\,u_0(R(0)\,x)$, conditions (H1)-(H2) can be rewritten as follows:

\medskip\noindent {\bf(H1')} There exist positive constants $C_1> C_2$ such that
\[
\label{eq:assumptionv}
\mathfrak B_1(x)=\left(C_1+|x|^{2+\beta-\gamma}\right)^{-\frac1{1-m}}\le v_0(x)\le \left(C_2+|x|^{2+\beta-\gamma}\right)^{-\frac1{1-m}}=\mathfrak B_2(x)
\]
for all $x \in\RR^d$.

\smallskip\noindent {\bf(H2')} There exist $C \in [C_2,C_1]$ and $f\in\LL^{1,\gamma}(\RR^d)$ such that
\[ \label{eq:assumptionmsmallermstarv}
v_0(x)=\mathfrak B(x)+f(x) \quad \forall\,x \in\RR^d\,.
\]
Assumption (H1') is nothing else than~\eqref{Ineq:sandwiched}. Note that, for greater readability, in (H2') we have replaced $ R(0)^{d-\gamma}\,f(R(0)\,x) $ with $f(x)$ again. As mentioned in Section~\ref{Sec:Asymptotic}, if $m>m_\ast$ the difference $\mathfrak B_2 -\,\mathfrak B_1 $ is in $\LL^{1,\gamma}(\RR^d)$, so that (H2') is implied by (H1') and the map $ C \mapsto \int_{\RR^d} (v_0-\mathfrak B)\,|x|^{-\gamma}\dx $ is continuous, monotone increasing and changes sign in $[C_2,C_1]$. Hence, in this case there exists a unique $C\in[C_2,C_1]$ such that
\[\label{eq: zero-rel-mass}
\int_{\RR^d}\left(v_0-\mathfrak B \right) \frac{\dx}{|x|^\gamma}=0\,.
\]
It is clear that, as a consequence of Proposition~\ref{prop:relconsmass}, under assumptions (H1')-(H2') the relative mass of $v$ is also conserved, that is
\[\label{eq: cons-rel-mass-v}
\int_{\RR^d} \left[v(t,x)-\mathfrak B(x) \right]\frac{\rd x}{|x|^\gamma}=\int_{\RR^d}\left(v_0-\mathfrak B \right)\frac{\rd x}{|x|^\gamma} \quad \forall\,t>0
\]
provided $\int_{\RR^d}|v_0-\mathfrak B|\,{|x|^{-\gamma}} \rd x $ is finite for some $C>0$. If $ m > m_\ast $, we deduce that
\[\label{eq: cons-rel-mass-v-zero}
\int_{\RR^d} \left[v(t,x)-\mathfrak B(x)\right] \frac{\rd x}{|x|^\gamma}=0 \quad \forall\,t > 0\,.
\]
On the other hand, if $ m \le m_\ast $ we cannot ensure that this identity still holds, but, nevertheless, the conservation of relative mass is still true and reads
\[
\int_{\RR^d} \left[v(t,x)-\mathfrak B(x) \right] \frac{\rd x}{|x|^\gamma}=\int_{\RR^d} f\,\frac{\rd x}{|x|^\gamma} \quad \forall\,t>0\,,
\]
where the r.h.s.~does not necessarily takes the value $0$.

\subsection{The relative error: a nonlinear Ornstein-Uhlenbeck equation}

Consider a solution $v$ of~\eqref{eq.FDE.FP} corresponding to an initial datum that satisfies (H1')-(H2'). As in~\cite[Section~2.3]{BBDGV},
let us introduce the ratio
\[\label{def.w}
w(t,x):=\frac{v(t,x)}{\mathfrak B(x)} \quad \forall\,(t,x) \in \RR^+ \times \RR^d\,.
\]
The difference $ w-1 $ is usually referred to as \emph{relative error} between $v$ and $\mathfrak B$. In view of~\eqref{eq.FDE.FP}, it is straightforward to check that $w$ is a solution to
\be{eq.FDE.OU}
\begin{cases}
|x|^{-\gamma}\,w_t=-\,\frac1{\mathfrak B}\,\nabla\cdot\Big(|x|^{-\beta}\,\mathfrak B\,w\,\nabla \left((w^{m-1}-1)\,\mathfrak B^{m-1}\right)\Big) & {\rm in} \ \RR^+ \times\RR^d\,, \\[6pt]
\displaystyle w(0,\cdot)=w_0:=\frac{v_0}{\mathfrak B} & {\rm in} \ \RR^d\,,
\end{cases}
\ee
which can be seen as a nonlinear, weighted equation of Ornstein-Uhlenbeck type. Let us also define the quantities
\[
W_1:=\inf_{x \in \RR^d} \frac{\mathfrak B_1(x)}{\mathfrak B(x)} \le \sup_{x\in\RR^d} \frac{\mathfrak B_2(x)}{\mathfrak B(x)}=: W_2\,.
\]
A straightforward calculation yields
\[
0 \ < \ W_1=\left(\frac{C}{C_1}\right)^{\frac1{1-m}} \ \le \ 1 \ \le \ \left(\frac{C}{C_2}\right)^{\frac1{1-m}}=W_2 \ < \infty\,.
\]
In terms of $w_0$, assumptions (H1') and (H2') can in turn be rewritten as follows:

\medskip\noindent {\bf(H1'')} There exist positive constants $C_1>C_2$ and $ C \in [C_2,C_1] $ such that
\[
W_1\le\frac{\mathfrak B_1(x)}{\mathfrak B(x)} \le w_0(x) \le \frac{\mathfrak B_2(x)}{\mathfrak B(x)} \le W_2 \quad \forall\,x \in \RR^d\,.
\]

\smallskip\noindent {\bf(H2'')} There exists $f\in \LL^{1,\gamma}(\RR^d)$ such that
\[
w_0(x)=1+\frac{f(x)}{\mathfrak B(x)} \quad \forall\,x \in \RR^d\,.
\]
Note that, as a consequence of the $\LL^{1,\gamma}$-contraction estimate and the comparison principle (Proposition~\ref{Prop:contraction} and Corollary~\ref{Lem:MP}), assumptions (H1'')-(H2'') are satisfied by a solution~$w(t)$ of~\eqref{eq.FDE.OU} at any $ t>0$ if they are satisfied by $w_0$, for a suitable $f$ depending also on $t$ (clearly the same holds for $v$, with respect to (H1')-(H2')).

\section{Regularity, relative uniform convergence and asymptotic rates}\label{Sec:Main}

The goal of this section is to show that the relative error $w(t)-1$ converges to zero \emph{uniformly} as $ t \to \infty$. Then in Section~\ref{sect: hp-rates}, by taking advantage of this result, we shall prove that such convergence occurs with explicit \emph{exponential rates}.\newpage

\subsection{Global regularity estimates: convergence without rates}\label{Sec:Regularity}

\subsubsection*{\bf a) From local to global estimates}\label{sect: glob-reg}
By exploiting similar scaling techniques as in~\cite[Section~2.4]{BBDGV}, we use the regularity results of the Appendix in order to get \emph{global} regularity estimates for $w$.
\begin{Lem}\label{lem:holdereg} Assume that $m\in(0,1)$. Let $ w $ be the solution of~\eqref{eq.FDE.OU} corresponding to an initial datum $w_0$ satisfying {\rm (H1'')-(H2'')}. Then there exist $\etanu \in(0,1) $ and a positive constant $ \mathcal K>0$, depending on $d$, $m$, $\beta$, $\gamma$, $C$, $C_1$, $C_2$ such that:
\be{eq: scaling-grad-v}
\left\| \nabla v(t) \right\|_{\LL^\infty(B_{2\lambda} \setminus B_{\lambda})} \le \frac{Q_1}{\lambda^{\frac{2+\beta-\gamma}{1-m}+1}} \quad \forall\,t \ge 1\,,\quad\forall\,\lambda>1\,,
\ee
\be{eq: est-k-glob}
\sup_{t\ge 1} \| w \|_{C^k( (t,t+1) \times B_{\varepsilon}^c)} < \infty \quad \forall\,k \in \mathbb{N}\,, \ \forall\,\varepsilon>0\,,
\ee
\be{eq: est-alfa-glob}
\sup_{t\ge 1} \|w(t)\|_{C^\etanu(\RR^d)} < \infty\,,
\ee
\be{eq:mathcalH}
\sup_{\tau\ge t} \left|w(\tau)-1 \right|_{C^\etanu(\RR^d)} \le \mathcal K\,\sup_{\tau\ge t} \left\| w(\tau)-1 \right\|_{\LL^\infty(\RR^d)} \quad \forall\,t\ge 1\,,
\ee
where $v$ is the solution of~\eqref{eq.FDE.FP} corresponding to the initial datum $ v_0=w_0\,\mathfrak B$.
\end{Lem}
\begin{proof} For all $\lambda>1$, let us consider the following rescaling:
\[
v_\lambda(t,x):=\lambda^{\frac{2+\beta-\gamma}{1-m}}\,v(t,\lambda\,x) \quad \forall\,(t,x) \in \RR^+ \times \RR^d\,.
\]
It is straightforward to check that $ v_\lambda $ satisfies the same equation as $ v$, with initial datum $ (v_0)_\lambda$. In particular, since $ (v_0)_\lambda $ is bounded and bounded away from zero in $ B_2 \setminus B_{\varepsilon/2} $ independently of $\lambda $ (consequence of (H1')), in view of standard parabolic regularity there holds
\[\label{eq: scaling-v-1}
\left | v_\lambda \right |_{C^k((t,t+1) \times (B_1 \setminus B_{\varepsilon}))} \le Q_k \quad \forall\,t \ge 1
\]
for all $k \in \mathbb{N} $ and some $ Q_k>0$ depending only on $d$, $m$, $\beta$, $\gamma$, $C$, $C_1$, $C_2$ and $\varepsilon$, but independent of $\lambda $. By undoing the scaling, this is equivalent to
\[\label{eq: scaling-v-2}
\left | v \right |_{C^k_x((t,t+1) \times (B_{\lambda} \setminus B_{\varepsilon \lambda}))} \le \frac{Q_k}{\lambda^{\frac{2+\beta-\gamma}{1-m}+k}}\,, \quad \left | v \right |_{C^k_\tau((t,t+1) \times (B_{\lambda} \setminus B_{\varepsilon \lambda}))} \le \frac{Q_k}{\lambda^{\frac{2+\beta-\gamma}{1-m}}} \quad \forall\,t \ge 1\,,
\]
where by $ C^k_x $ and $ C^k_\tau $ we mean partial derivatives restricted to space and time, respectively. As a special case, this proves~\eqref{eq: scaling-grad-v} and~\eqref{eq: est-k-glob} upon observing that
\be{eq: scaling-B-est}
\left\| \mathfrak B^{-1} \right\|_{\LL^\infty(B_{\lambda} \setminus B_{\varepsilon \lambda})} \approx \lambda^{\frac{2+\beta-\gamma}{1-m}}\,, \quad \left | \mathfrak B \right |_{C^k(B_{\lambda} \setminus B_{\varepsilon \lambda})} \approx \lambda^{-\frac{2+\beta-\gamma}{1-m}-k}\,.
\ee
As for proving~\eqref{eq: est-alfa-glob}, it is enough to notice that, for some $\etanu \in (0,1) $ and another constant $ Q_\etanu>0$ depending on the same quantities as $ Q_k $, with the exception of $\varepsilon $, the estimate
\[\label{eq: scaling-v-3}
\left| v(t) \right|_{C^\etanu(B_{1/2})} \le Q_\etanu \quad \forall\,t \ge 1
\]
follows from the regularity results of the Appendix. Indeed, since $ v $ is bounded and bounded away from zero in $\RR^+ \times B_1$, we can apply Corollary~\ref{Thm.C.alpha.lin-2} with the choices
\[\label{eq: scaling-v-2-ch}
a(t,x)=(1-m)\,v^{m-1}(t,x)\,, \quad B(t,x)=-\,\frac{2+\beta-\gamma}{1-m}\,x\,|x|^{-\gamma}\,v^{1-m}(t,x)\,.
\]

\smallskip As for~\eqref{eq:mathcalH}, let $ z(t,x):=v(t,x)-\mathfrak B(x)$. Straightforward computations show that the equation solved by $z$ reads
\[\label{eq: eq-h-1}
|x|^{-\gamma}\,z_t=\nabla \cdot \left[|x|^{-\beta}\,a(t,x) \left(\nabla z + B(t,x)\,z \right) \right]\,,
\]
with the same function $a$ as above and
\[\label{eq: eq-h-coeff}
B(t,x)=v^{1-m}(t,x) \left( \frac{v^{m-1}(t,x) -\,\mathfrak B^{m-1}(x)}{v(t,x)-\mathfrak B(x)} -\,\mathfrak B^{m-2}(x) \right) \nabla \mathfrak B(x)\,.
\]
We are therefore again in position to use Corollary~\ref{Thm.C.alpha.lin-2} to get 
\[\label{eq: eq-h-2}
| z(t) |_{C^\etanu(B_{2^{-1/\alpha}})} \le \mathcal K \left\| z \right\|_{\LL^\infty ((t+2, t+3) \times B_{2^{2/\alpha}})} \quad \forall\,t \ge 1
\]
for some $ \mathcal K>0$ depending on $d$, $m$, $\beta$, $\gamma$, $C$, $C_1$, $C_2$. From here on $ \mathcal K $ will denote a general positive constant, which may change from line to line. Corollary~\ref{Thm.C.alpha.lin-2} holds with inessential modifications if one replaces balls with annuli. By performing scalings, we deduce that
\[\label{eq: eq-h-3}
| z(t) |_{C^\etanu(B_{\lambda} \setminus B_{\lambda/2})} \le \mathcal K\,\lambda^{-\etanu} \left\| z \right\|_{\LL^\infty ((t+2, t+3) \times (B_{8\lambda} \setminus B_{2\lambda}) )} \quad \forall\,t \ge 1\,.
\]
By standard computations and~\eqref{eq: scaling-B-est}, which holds even if $ k $ is not an integer, and the identity $ w(t)-1=z(t)/\mathfrak B $, we obtain
\[\label{eq: eq-h-4}
\begin{aligned}
 &\hspace*{-24pt} |w(t)-1|_{C^\etanu(B_{\lambda} \setminus B_{\lambda/2})} \\
\le &\, \left\| \mathfrak B^{-1} \right\|_{\LL^\infty(B_{\lambda} \setminus B_{\lambda/2})} |z(t)|_{C^\etanu(B_{\lambda} \setminus B_{\lambda/2})} \\
&\quad + \left\| \mathfrak B^{-1} \right\|_{\LL^\infty(B_{\lambda} \setminus B_{\lambda/2})}^2 |\mathfrak B|_{C^\etanu(B_{\lambda} \setminus B_{\lambda/2})}\,\| z(t) \|_{\LL^\infty(B_{\lambda} \setminus B_{\lambda/2})} \\
\le &\, \mathcal K\,\lambda^{\frac{2+\beta-\gamma}{1-m}-\etanu} \left( \left\| z \right\|_{\LL^\infty ((t+2, t+3) \times (B_{8\lambda} \setminus B_{2\lambda}) )} + \| z(t) \|_{\LL^\infty(B_{\lambda} \setminus B_{\lambda/2})} \right) \\
 \le &\, \mathcal K\,\lambda^{-\etanu} \left\| w-1 \right\|_{\LL^\infty ((t, t+3) \times (B_{8\lambda} \setminus B_{2\lambda}) )}.
\end{aligned}
\]
By taking $\lambda=2^{j/\alpha}$, with $j\ge -1$ integer, we conclude the proof of~\eqref{eq:mathcalH}. \end{proof}

\begin{Rem}\label{rem: reg} It is important to point out that the applicability of the results of the Appendix with coefficients $a$ and $B$ as in the proof of Lemma~\ref{lem:holdereg} relies on the fact that, under assumptions (H1')-(H2'), $v$ is locally bounded and bounded away from zero. In other words, we do not claim that we have a Harnack inequality for general solutions to the \emph{degenerate/singular} Equation~\eqref{FD-FP}.\end{Rem}

\subsubsection*{\bf b) The relative free energy and Fisher information} \label{sect: rig-ent}
By proceeding along the lines of~\cite[Section~2.5]{BBDGV}, we redefine the relative \emph{free energy} functional as
\[\label{Entropy}
\mathcal{F}[w]=\frac1{m-1}\int_{\RR^d}\left[w^m-1-m\,(w-1)\right]\frac{\mathfrak B^m}{|x|^\gamma}\dx\,,
\]
with a slight abuse of notations in the sense that we consider it as a functional acting on $w=v/\mathfrak B$. Again, if we formally derive $\mathcal{F}[w(t)]$ with respect to $t$ along the flow~\eqref{eq.FDE.OU} we obtain
\be{Entropy.Prod}
\frac{\rd}{\dt}\,\mathcal{F}[w(t)]=-\,\frac{m}{1-m}\,\I[w(t)]\,,
\ee
where $\I $ is the relative \emph{Fisher information}, redefined in terms of $w$ as
\[\label{Fisher.Info}
\I[w]=\int_{\RR^d} w\left|\nabla\left(\left(w^{m-1}-1\right)\mathfrak B^{m-1}\right)\right|^2\mathfrak B\,\frac{dx}{|x|^\beta}\,.
\]
However, the rigorous justification of~\eqref{Entropy.Prod} is not straightforward, and to this end we need to take advantage of the global regularity estimates provided in Section~\ref{sect: glob-reg}.
\begin{Prop}[Entropy-entropy production identity]\label{Prop:EntrProd} Assume that $m\in(0,1)$. If~$w$ is a solution of~\eqref{eq.FDE.OU} corresponding to an initial datum $w_0$ satisfying assumptions {\rm (H1'')-(H2'')}, then the free energy~$\mathcal{F}[w(t)]$ is finite for all $t \ge 0$ and identity~\eqref{Entropy.Prod} holds.
\end{Prop}
\begin{proof} We proceed through three steps, following the lines of proof of~\cite[Proposition~2]{BBDGV}. We skip the proof of the fact that $\mathcal{F}[w(t)] $ is finite, since it goes exactly as in~\cite[Lemma~4]{BBDGV}, with inessential modifications. For the sake of greater readability we shall omit time-dependence, at least when this does not compromise comprehension.

\smallskip\noindent$\bullet~$\textsc{Step 1.} Consider the same cut-off function $\phi_\lambda $ as in the proof of Proposition~\ref{prop:relconsmass}, with $\lambda > 1$. Then, by using~\eqref{eq.FDE.OU}, the identity $ w\,\mathfrak B=v $ and integrating by parts, we obtain
\[\label{eq: entropy-deriv}
\begin{aligned}
&\hspace{-24pt} -\,\frac{\rd}{\dt}\,\frac1{1-m}\int_{\RR^d}\left[w^m-1 - m\,(w-1) \right] \phi_\lambda\,\frac{\mathfrak B^m}{|x|^\gamma}\dx \\
=& -\frac m{1-m} \int_{\RR^d} w\,\nabla \left( v^{m-1}-\mathfrak B^{m-1} \right) \cdot \nabla \left[(v^{m-1}-\mathfrak B^{m-1})\,\phi_\lambda\,\right] \mathfrak B\,\frac{dx}{|x|^\beta} \\
=& -\frac m{1-m} \int_{\RR^d} w \left| \nabla \left(v^{m-1}-\mathfrak B^{m-1}\right) \right|^2 \phi_\lambda\,\mathfrak B\,\frac{dx}{|x|^\beta} + \frac m{2\,(1-m)}\,\mathcal{R}(\lambda)
\end{aligned}
\]
where
\begin{multline*}
\hspace*{-8pt}\mathcal{R}(\lambda):=-\int_{\RR^d} \nabla \left[(v^{m-1}-\mathfrak B^{m-1})^2\right] \cdot \nabla\phi_\lambda\,v\,\frac{\dx}{|x|^\beta} \\
=\int_{B_{2\lambda} \setminus B_\lambda} \left|v^{m-1}-\mathfrak B^{m-1}\right|^2 \left(\nabla v\cdot\nabla\phi_\lambda + v\,\Delta \phi_\lambda -\,\beta\,v\,\frac x{|x|^2}\cdot \nabla\phi_\lambda \right) \frac{\dx}{|x|^\beta}\,.
\end{multline*}
We have
\begin{multline*}\label{R.lambda}
\frac{\left|\mathcal{R}(\lambda)\right|}{(1-m)^2}\le\int_{B_{2\lambda} \setminus B_\lambda} \! \left| v-\mathfrak B \right|^2 \mathfrak B_1^{2\,(m-2)} \! \left(|\nabla v|\,|\nabla\phi_\lambda| + v\,|\Delta\phi_\lambda| +\tfrac\beta{|x|}\,v\,|\nabla\phi_\lambda| \right) \! \frac{\dx}{|x|^\beta} \\
\le\,c\int_{B_{2\lambda} \setminus B_\lambda} \left| v -\,\mathfrak B \right| \frac{\dx}{|x|^\gamma}
\end{multline*}
where, in the last step, we used the inequality
\be{R.lambda.2}
\left|v-\mathfrak B\right| \mathfrak B_1^{2\,(m-2)} \left(|\nabla v|\,|\nabla\phi_\lambda|+v\,|\Delta\phi_\lambda|+\tfrac\beta{|x|}\,v\,|\nabla\phi_\lambda|\right)\frac 1{|x|^\beta} \le \frac{c}{|x|^\gamma} \quad \forall\,x \in B_{2\lambda} \setminus B_\lambda
\ee
for some $c>0$, independent of $\lambda > 1$. We shall establish~\eqref{R.lambda.2} in Step 2. By assumptions (H1')-(H2') and the $\LL^{1,\gamma}$-contraction principle, the difference $v-\mathfrak B$ is in $\LL^{1,\gamma}(\RR^d)$, so that $\lim_{\lambda\to\infty} \mathcal{R}(\lambda)=0$ and the proof is completed by passing to the limit as $\lambda\to\infty$.

\smallskip\noindent$\bullet~$\textsc{Step 2.} Recalling~\eqref{eq: scaling-grad-v}, we know that
\[
\left\| \nabla v(t) \right\|_{\LL^\infty(B_{2\lambda} \setminus B_{\lambda})} \le \frac{Q_1}{\lambda^{\frac{2+\beta-\gamma}{1-m}+1}} \quad \forall\,t \ge 1\,.
\]
Moreover, since $v$ is trapped between two Barenblatt profiles, we have
\[\label{eq: scaling-grad-v-2}
\left\| v(t) \right\|_{\LL^\infty(B_{2\lambda} \setminus B_{\lambda})} \le \frac{Q_0}{\lambda^{\frac{2+\beta-\gamma}{1-m}}}\,, \quad \left\| v(t) -\,\mathfrak B \right\|_{\LL^\infty(B_{2\lambda} \setminus B_{\lambda})} \le \frac{Q_2}{\lambda^{\frac{(2+\beta-\gamma)\,(2-m)}{1-m}}}\,.
\]
The estimates hold for suitable positive constants $ Q_0$, $ Q_1 $ and $ Q_2 $ which are all independent of $\lambda$, $t>1$. As for $\phi_\lambda$, by construction we have that
\[\label{est.grad.lapl.varphi-est}
\left\| \nabla\phi_\lambda \right\|_{\LL^\infty(B_{2\lambda} \setminus B_{\lambda})} \le \frac{c_1}{\lambda}\,, \quad \left\| \Delta \phi_\lambda \right\|_{\LL^\infty(B_{2\lambda} \setminus B_{\lambda})} \le \frac{c_2}{\lambda^2}\,,
\]
for some $ c_1,c_2>0$  which are also independent of $\lambda$. Estimate~\eqref{R.lambda.2} readily follows.

\smallskip\noindent$\bullet~$\textsc{Step 3.} It remains to take care of the origin. In principle solutions are only H\"older regular (see the Appendix). Nevertheless, since $v$ is uniformly bounded and locally bounded away from zero, standard energy estimates (see again~\cite{VazBook} as a general reference) ensure, \emph{e.g.},~that the quantities
$\int_{t_1}^{t_2} \int_{B_1} \left| \nabla v^m \right|^2 \frac{\dx}{|x|^\beta} \dt$ and $\int_{t_1}^{t_2} \int_{B_1} \big| \big( v^{\frac{m+1}2}\big)_t \big|^2 \frac{\dx}{|x|^\gamma} \dt$
are finite for all $t_1,t_2>0$, which is enough in order to give sense to~\eqref{Entropy.Prod} at least in a $\LL^1_{\rm loc}(\RR^+) $ sense.
\end{proof}

\subsubsection*{\bf c) Uniform convergence in relative error}\label{sect-pcre}
By mimicking the proofs of~\cite[Lemma~5 and Corollary~1]{BBDGV}, we can show that the rescaled solution $v$ converges to~$\mathfrak B$ uniformly in the strong sense of the relative error.
\begin{Prop}[Convergence in relative error without rates]\label{prop: cvrgcewithoutrate} Assume that $m\in(0,1)$. If~$w$ is a solution of~\eqref{eq.FDE.OU} corresponding to an initial datum $w_0$ satisfying assumptions {\rm (H1'')-(H2'')}, then
\[\label{eq: conv-rel}
\lim_{t\to\infty} \left\| w(t)-1 \right\|_{\LL^\infty(\RR^d)}=0\,.
\]
\end{Prop}
\begin{proof} For all $\tau>0$, we set $w_\tau(t,x):=w(t+\tau,x)$. In view of~~\eqref{eq: est-k-glob}, there exists a sequence $\tau_n \to \infty $ such that $ w_{\tau_n} $ converges locally uniformly in $ (1,\infty) \times (\RR^d \setminus \{ 0 \} ) $ to some $ w_\infty \in \LL^\infty((1,\infty) \times \RR^d)$. Moreover, by Corollary~\ref{Lem:MP} we deduce that
\be{eq: conv-rel-proof-2}
0 < W_1 \le \frac{\mathfrak B_1(x)}{\mathfrak B(x)} \le w_\infty(t,x) \le \frac{\mathfrak B_2(x)}{\mathfrak B(x)} \le W_2 < \infty \quad \forall\,(t,x) \in (0,\infty) \times \RR^d\,.
\ee
Thanks to Proposition~\ref{Prop:EntrProd}, there holds
\[\label{eq: conv-rel-proof-1}
\mathcal{F}[w(\tau_n+1)]-\mathcal{F}[w(\tau_n+2)]=\int_{\tau_n+1}^{\tau_n+2} \I[w(t)] \dt=\int_1^2 \I[w(t+\tau_n)] \dt \ge 0\,.
\]
Since $\mathcal{F}[w(s_n+2)] $ is bounded from below (as a consequence of (H1'')-(H2''), see again~\cite[Lemma~4]{BBDGV}), we infer that $\I[ w_{\tau_n}(t)]$ converges to zero in $\LL^1((1,2)) $ as $ n \to \infty$, that is,
\[\label{eq: conv-rel-proof-3}
\lim_{n \to \infty} \int_1^2 \int_{\RR^d} w_{\tau_n}(t,x) \left|\nabla\left(\left(w_{\tau_n}^{m-1}(t,x)-1\right) \mathfrak B^{m-1} \right)\right|^2 \mathfrak B\,\frac{dx}{|x|^\beta} \dt=0\,.
\]
By Fatou's lemma, this implies
\[\label{eq: conv-rel-proof-4}
\int_1^2 \int_{\RR^d} \liminf_{n \to \infty} \left|\nabla\left(\left(w_{\tau_n}^{m-1}(t,x)-1\right) \mathfrak B^{m-1} \right)\right|^2 w_{\infty}(t,x)\,\mathfrak B\,\frac{dx}{|x|^\beta} \dt=0\,.
\]
Still as a consequence of~\eqref{eq: est-k-glob} and~\eqref{eq: conv-rel-proof-2} we have that
\[
0\equiv\liminf_{n \to \infty} \left|\nabla\left(\left(w_{\tau_n}^{m-1}-1\right) \mathfrak B^{m-1} \right)\right|=\left|\nabla\left(\left(w_{\infty}^{m-1}-1\right) \mathfrak B^{m-1} \right)\right|\quad\mbox{a.e.}\quad(1,2) \times \RR^d\,.
\]
This means that the function $\left(w_{\infty}^{m-1}-1\right) \mathfrak B^{m-1} $ is constant, hence
\[
w_\infty(t,x)=\left[1+c(t)\,\mathfrak B^{1-m}(x)\right]^{-\frac1{1-m}} \quad \forall\,(t,x) \in (1,2) \times \RR^d\,.
\]
It is readily seen that the only possibility is $ c \equiv 1$. Indeed, if $ m > m_\ast $ this is due to the conservation of relative mass (Proposition~\ref{prop:relconsmass}), while in the case $ m \le m_\ast $ it is a consequence of the $\LL^{1,\gamma} $-contraction principle (Proposition~\ref{Prop:contraction}). Since we can repeat the same argument as above, up to subsequences, along \emph{any} sequence $\tau_n \to \infty$, in fact we have shown that
\[\label{eq: conv-rel-proof-5}
\lim_{t \to \infty} w(t) \equiv 1 \quad\mbox{in}\quad\LL^\infty_{\rm loc}(\RR^d)\,.
\]
In order to obtain the global uniform convergence, it is enough to recall~\eqref{eq: conv-rel-proof-2} and note that by dominated convergence we have $\lim_{t \to \infty} \| w(t)-1 \|_{\LL^p(\RR^d)}=0$ for all $p > d/(2+\beta-\gamma) $: the global $C^\etanu$ estimate,~\eqref{eq: est-alfa-glob}, and a standard interpolation like~\cite[Proof~of Theorem~1]{BBDGV} allow us to conclude.\end{proof}

\subsection{Hardy-Poincar\'e inequalities: convergence with rates}\label{sect: hp-rates}

As in~\cite{BBDGV,BDGV}, if $ m \neq m_\ast$, sharp rates of convergence towards the Barenblatt profile $\mathfrak B$ are related to the optimal constant $\Lambda>0$ of the Hardy-Poincar\'e-type inequality
\be{HP-ineq}
\int_{\RR^d}\left|\nabla f\right|^2\mathfrak B\,\frac{dx}{|x|^\beta} \ge \Lambda \int_{\RR^d}{\left|f\right|^2}\,\mathfrak B^{2-m}\,\frac{dx}{|x|^\gamma}
\ee
for any function $f\in C^\infty_c(\RR^d)$ such that, additionally, $\int_{\RR^d} f\,\mathfrak B^{2-m}\,\frac{dx}{|x|^\gamma}=0$ whenever $\int_{\RR^d} \mathfrak B^{2-m}\,\,\frac{dx}{|x|^\gamma}$ is finite, that is, for $ m>m_\ast$. The explicit value of $\Lambda $ has been computed explicitly in~\cite{BDMN2016a}, and is provided in Proposition~\ref{Prop:Spectrum}.

\subsubsection*{$\bullet$ Weighted linearization}
In order to better understand the asymptotic behaviour of the solutions at hand, let us outline our strategy. The idea, as in~\cite[Section~3.3]{BBDGV}, is to linearize the equation of the relative error~\eqref{eq.FDE.OU} around the equilibrium, by introducing a convenient weight. More precisely, let $f$ be such that
\[\label{g.Linearization}
w(t,x)=1+\varepsilon\,\frac{f(t,x)}{\mathfrak B^{m-1}(x)} \quad \forall\,(t,x) \in \RR^+ \times \RR^d\,,
\]
for some small $\varepsilon>0$. By substituting this expression in~\eqref{eq.FDE.OU} and neglecting higher order terms in $\varepsilon$ as $\varepsilon\to 0$, we formally obtain a \emph{linear} equation for $f$,
\be{Linearised.FP}
f_t=(1-m)\,|x|^\gamma\,\mathfrak B^{m-2}\,\nabla \cdot \left(|x|^{-\beta}\,\mathfrak B\,\nabla f \right)\,,
\ee
where the r.h.s.~ involves a positive, self-adjoint operator on $\LL^2(\RR^d,\mathfrak B^{2-m}\,|x|^{-\gamma}\dx) $ associated with the closure of the quadratic form defined by
\[\label{Fisher.Info.Lin}
\IL[\phi]:=(1-m) \int_{\RR^d} \left|\nabla \phi \right|^2 \mathfrak B\,\frac{dx}{|x|^\beta} \quad \forall\,\phi \in C^\infty_0(\RR^d)\,.
\]
The functional $\IL[\phi]$ is the linearized version of the Fisher information $\mathcal I$, divided by $ (1-m)$. By means of the same heuristics, we can linearize the free energy~$\mathcal F$ as well to get, up to a factor $1/m$,
\[
\FL[\phi]:=\frac12 \int_{\RR^d} \phi^2\,\mathfrak B^{2-m}\,\frac{dx}{|x|^\gamma}\,.
\]
If $f$ is a solution of~\eqref{Linearised.FP} then it is straightforward to infer that it satisfies
\be{eq: lin-eq-FI}
\frac{\rd}{\dt}\,\FL[f(t)]=-\,\IL[f(t)]\,,
\ee
which by the way could also have been obtained by linearizing~\eqref{Entropy.Prod}. In the case $ m>m_\ast $ the conservation of relative mass becomes, after linearization,
\[\label{eq: cons-mass-lin}
\int_{\RR^d} f(t,x)\,\mathfrak B^{2-m}\,\frac{dx}{|x|^\gamma}=0 \quad \forall\,t \ge 0\,.
\]
Hence, as a consequence of~\eqref{HP-ineq} and~\eqref{eq: lin-eq-FI}, we formally get the following exponential decay for the linearized free energy:
\[\label{Exp.Decay.Linear}
\FL[f(t)] \le e^{-\,2\,(1-m)\,\Lambda\,t}\,\FL[f(0)] \quad \forall\,t \ge 0\,.
\]

\subsubsection*{$\bullet$ Comparing linear and nonlinear quantities}\label{sect: lin-nonlin}
Our aim here is to proceed in a similar way as in~\cite[Sections~5 and~6.2]{BBDGV} so as to compare the free energy and Fisher information $\mathcal{F} $ and $\I $ with their linearized versions $\FL $ and $\IL$, respectively. This will then allow us to give a rigorous justification of the above exponential decay and to use such an information to infer a precise exponential decay for the relative error.

\smallskip Let us consider
\[
g=(w-1)\,\mathfrak B^{m-1}\,.
\]
For $t_0\ge0$ large enough, we deduce from Proposition~\ref{prop: cvrgcewithoutrate} the existence of $h\in(0,1/4)$ such that $\nrm{w(t)-1}\infty\le h$ for any $t\ge t_0$. The next result, whose proof we omit since it is identical to the one of~\cite[Lemma~3]{BBDGV}, shows the free energy compares with the \emph{linearized} free energy.
\begin{Lem}\label{Lem.Bounds.RE} Assume that $m\in(0,1)$. If~$w$ is a solution of~\eqref{eq.FDE.OU} corresponding to an initial datum $w_0$ satisfying assumptions {\rm (H1'')-(H2'')}, then there exists $t_0\ge0$ such that
\[\label{Disug.Entr.Lin-Nolin}
m\,(1+h)^{m-2}\,\FL[w(t)] \le\mathcal{F}[w(t)]\le m\,(1-h)^{m-2}\,\FL[w(t)] \quad \forall\,t \ge t_0\,.
\]
\end{Lem}
For simplicity we shall assume that $t_0=0$ from now on. We now state the analogue of~\cite[Lemma~5.4]{BGVm*}. The proof is again identical to the one performed in the case $(\beta,\gamma)=(0,0)$, so we skip it.
\begin{Lem}\label{Lem.Entr.Lp}
Assume that $m\in(0,1)$. If~$w$  is a solution of~\eqref{eq.FDE.OU} corresponding to an initial datum $w_0$ satisfying assumptions {\rm (H1'')-(H2'')}, then
\[\label{Disug.Entr.Lp}
\left\|w(t) - 1\right\|^{\frac{2-m}{1-m}}_{\LL^{\frac{2-m}{1-m},\gamma}(\RR^d)} \le \mathcal K\,\FL[w(t)] \quad \forall\,t \ge 0\,,
\]
where $\mathcal K$ is a positive constant depending only on $ m$, $C_1$, $C_2$.
\end{Lem}
The next step is to get a bound of the $\LL^\infty$ norm of the relative error in terms of the free energy.
\begin{Lem}\label{lem.EntrVsLp}
Assume that $m\in(0,1)$. If~$w$  is a solution of~\eqref{eq.FDE.OU} corresponding to an initial datum $w_0$ satisfying assumptions {\rm (H1'')-(H2'')}, then the following estimates hold:
\be{Entr.Lp.gamma}
\left\|w(t) - 1\right\|^{\frac{2-m}{1-m}}_{\LL^{\frac{2-m}{1-m},\gamma}(\RR^d)} \le \kappa_0\,\mathcal{F}[w(t)] \le \kappa_0\,\mathcal{F}[w_0] \quad \forall\,t \ge 0
\ee
and
\be{interp.Calpha.Lp}
\sup_{\tau \ge t} \left\| w(\tau)-1 \right\|_{\LL^\infty(\RR^d)} \le \kappa_\infty\,\sup_{\tau \ge t} \mathcal{F}[w(t)]^{\varth} \le \kappa_\infty\,\mathcal{F}[w_0]^{\varth} \quad \forall\,t \ge 1\,,
\ee
where
\[\label{eq: def-theta-exp}
\varth:=
\begin{cases}
\frac{(1-m)\,(2+\beta-\gamma)}{(1-m)\,(2+\beta)+2+\beta-\gamma} & \textrm{if } \gamma \in (0,d)\,, \\[6pt]
\frac{1-m}{2-m} & \textrm{if } \gamma \le 0\,,
\end{cases}
\]
and the positive constants $\kappa_0$ and $\kappa_\infty$ depend on $ d$, $m$, $\gamma$, $\beta$, $ C_1$, $ C$, $ C_2$.
\end{Lem}
\begin{proof} Estimate~\eqref{Entr.Lp.gamma} is a direct consequence of Lemmas~\ref{Lem.Bounds.RE}-\ref{Lem.Entr.Lp} and of the fact that the free energy is nonincreasing by~\eqref{Entropy.Prod}.

\smallskip As for~\eqref{interp.Calpha.Lp}, let us first consider the case $\gamma \ge 0$. In this range we deduce from (H1'') that
\[\label{est-w-B}
\left| w(t,x) - 1 \right| \le \mathcal C\,\mathfrak B^{1-m}(x) \quad \forall\,(t,x) \in \RR^+ \times \RR^d
\]
for a constant $\mathcal C>0$ depending on $ m$, $C_1$ and $C_2$, and, as a consequence,
\be{eq: ineq-w-gamma}
\left| w(t,x) - 1 \right|^{\frac\gamma{2+\beta-\gamma}} \le \frac{\kappa_1}{|x|^\gamma} \quad \forall\,(t,x) \in \RR^+ \times \RR^d
\ee
for some $\kappa_1 $ depending on $ m$, $\gamma$, $\beta$, $C_1$, $C_2$. By combining~\eqref{Entr.Lp.gamma} with~\eqref{eq: ineq-w-gamma} we deduce that
\be{Entr.Lp.gamma-2}
\left\| w(t)-1 \right\|^{\frac{2-m}{1-m}+\frac\gamma{2+\beta-\gamma}}_{\LL^{\frac{2-m}{1-m}+\frac\gamma{2+\beta-\gamma}}(\RR^d)}=\left\| w(t)-1 \right\|^{\frac1{\varth}}_{\LL^{\frac1{\varth}}(\RR^d)}
\le \kappa_1\,\kappa_0\,\mathcal{F}[w(t)] \le \kappa_1\,\kappa_0\,\mathcal{F}[w_0]
\ee
for all $t \ge 0$. Hence,~\eqref{interp.Calpha.Lp} follows with $\kappa_\infty=(\kappa_1\,\kappa_0)^\varth\,\mathcal{C}_{\etanu,0,1/\varth}^{1/{d}}\,\mathcal K$ by using~\eqref{eq:mathcalH} and generalised interpolation inequalities due to Gagliardo and Nirenberg (see, \emph{e.g.},~\cite[Section~3]{BBDGV} or~\cite[Appendix~A.3]{BGVm*}):
\be{GN-interp}
\begin{aligned}
\sup_{\tau \ge t} \left\| w(\tau)-1 \right\|_{\LL^\infty(\RR^d)} \le &\, \mathcal{C}_{\etanu,0,1/\varth}^{\frac{d}{\varth d + \etanu}}\,\sup_{\tau \ge t} \left| w(\tau)-1 \right|_{C^\etanu(\RR^d)}^{\frac{\varth d}{\varth d + \etanu}} \sup_{\tau \ge t} \left\| w(\tau)-1 \right\|_{\LL^{\frac1{\varth}}(\RR^d)}^{\frac{\etanu}{\varth d + \etanu}} \\
\le &\, \mathcal{C}_{\etanu,0,1/\varth}^{\frac{d}{\varth d + \etanu}}\,K^{\frac{\varth d}{\varth d + \etanu}} \sup_{\tau\ge t} \left\| w(\tau)-1 \right\|_{\LL^\infty(\RR^d)}^{\frac{\varth d}{\varth d + \etanu}} \sup_{\tau \ge t} \left\| w(\tau)-1 \right\|_{\LL^{\frac1{\varth}}(\RR^d)}^{\frac{\etanu}{\varth d + \etanu}}
\end{aligned}
\ee
for all $t \ge 1$, where $\mathcal{C}_{\etanu,0,1/\varth} $ is a positive constant depending only on $ d$, $\etanu$, $\varth$.

Let us now deal with the case $\gamma < 0$, where inequality~\eqref{eq: ineq-w-gamma} is no longer valid, so we have to proceed in a different way. To this end, first of all note that by H\"older's interpolation we obtain
\[
\| w(t)-1 \|_{\LL^p(B_r)} \le \left( \int_{B_r} \left| w(t,x)-1 \right|^{\frac{2-m}{1-m}} \! \frac{\rd x}{|x|^\gamma} \right)^{\frac{1-m}{2-m}} \! \left( \int_{B_r} |x|^{\gamma\,\frac{p(1-m)}{2-m-p(1-m)}}\,\rd x \right)^{\frac1p-\frac{1-m}{2-m}}
\]
and
\[
\| w(t)-1 \|_{\LL^q(B_r^c)} \le \left( \int_{B_r^c} \left| w(t,x)-1 \right|^{\frac{2-m}{1-m}} \! \frac{\rd x}{|x|^\gamma} \right)^{\frac{1-m}{2-m}} \! \left( \int_{B_r^c} |x|^{\gamma\,\frac{q(1-m)}{2-m-q(1-m)}}\,\rd x \right)^{\frac1q-\frac{1-m}{2-m}}
\]
for all $r>0$ and $p$, $q\in(0,\frac{2-m}{1-m})$.
In particular, in view of Lemma~\ref{Lem.Entr.Lp}, there exist~$p$ (sufficiently close to $0$), $ q $ (sufficiently close to $\tfrac{2-m}{1-m}$) and a positive constant $D$ depending on $d$, $m$, $\gamma$, $C_1$, $C_2$, $r$, such that
\be{est: gamma-pos-1}
\| w(t)-1 \|_{\LL^p(B_r)} \le D\,\mathcal{F}[w(t)]^{\frac{1-m}{2-m}} \quad \textrm{and} \quad \| w(t)-1 \|_{\LL^q(B_r^c)} \le D\,\mathcal{F}[w(t)]^{\frac{1-m}{2-m}}\,.
\ee
Let $\phi_\lambda $ be the same family of cut-off functions as in the proof of Proposition~\ref{prop:relconsmass}. It is clear that
\[\label{est: gamma-pos-2}
\left| \phi_2 \left( w(t)-1 \right) \right|_{C^\etanu(\RR^d)} \le c \left( \left| w(t)-1 \right|_{C^\etanu(\RR^d)} + \left\| w(t)-1 \right\|_{\LL^\infty(\RR^d)} \right)
\]
and
\[\label{est: gamma-pos-2-bis}
\left| (1-\phi_1) \left( w(t)-1 \right) \right|_{C^\etanu(\RR^d)} \le c \left( \left| w(t)-1 \right|_{C^\etanu(\RR^d)} + \left\| w(t)-1 \right\|_{\LL^\infty(\RR^d)} \right)
\]
for some $c>0$ depending only on $\etanu $ and $\phi$. Thanks to~\eqref{eq:mathcalH}, by applying~\eqref{GN-interp} to the functions $\phi_2 \left( w(t)-1 \right) $ and $(1-\phi_1) \left( w(t)-1 \right)$, we obtain
\begin{multline*}
\sup_{\tau \ge t} \left\| w(\tau)-1 \right\|_{\LL^\infty(B_2)} \\
\le\,\mathcal{C}_{\etanu,0,p}^{\frac{d p}{d + \etanu p}}\,c^{\frac{d}{ d + \etanu p}}\,(\mathcal K+1)^{\frac{d}{d + \etanu p}}\,\sup_{\tau \ge t} \left\| w(\tau)-1 \right\|_{\LL^\infty(\RR^d)}^{\frac{ d}{ d + \etanu p}}\,\sup_{\tau \ge t} \left\| w(\tau)-1 \right\|_{\LL^{p}(B_4)}^{\frac{\etanu p}{d + \etanu p}}
\end{multline*}
and
\begin{multline*}\label{est: gamma-pos-4}
\sup_{\tau \ge t} \left\| w(\tau)-1 \right\|_{\LL^\infty(B_2^c)} \\
\le\,\mathcal{C}_{\etanu,0,q}^{\frac{d q}{d + \etanu q}}\,c^{\frac{d}{ d + \etanu q}}\,(\mathcal K+1)^{\frac{d}{d + \etanu q}}\,\sup_{\tau \ge t} \left\| w(\tau)-1 \right\|_{\LL^\infty(\RR^d)}^{\frac{ d}{ d + \etanu q}}\,\sup_{\tau \ge t} \left\| w(\tau)-1 \right\|_{\LL^{q}(B_1^c)}^{\frac{\etanu q}{d + \etanu q}}
\end{multline*}
for all $t \ge 1$. Hence, by exploiting~\eqref{est: gamma-pos-1} with $r=4$ and $r=1$ in the right-hand sides and summing up the two estimates, we end up with
\[\label{est: gamma-pos-5}
\begin{aligned}
& \hspace*{-12pt}\sup_{\tau \ge t} \left\| w(\tau)-1 \right\|_{\LL^\infty(\RR^d)} \\
\le &\, \mathcal{C}_{\etanu,0,p}^{\frac{d p}{d + \etanu p}}\,c^{\frac{d}{ d + \etanu p}}\,(\mathcal K+1)^{\frac{d}{d + \etanu p}}\,D^{\frac{\etanu p}{d + \etanu p}}\,\sup_{\tau \ge t} \left\| w(\tau)-1 \right\|_{\LL^\infty(\RR^d)}^{\frac{ d}{ d + \etanu p}}\,\sup_{\tau \ge t} \mathcal{F}[w(\tau)]^{\frac{1-m}{2-m} \frac{\etanu p}{d + \etanu p} } \\
 &\, + \mathcal{C}_{\etanu,0,q}^{\frac{d q}{d + \etanu q}}\,c^{\frac{d}{ d + \etanu q}}\,(\mathcal K+1)^{\frac{d}{d + \etanu q}}\,D^{\frac{\etanu q}{d + \etanu q}}\,\sup_{\tau \ge t} \left\| w(\tau)-1 \right\|_{\LL^\infty(\RR^d)}^{\frac{ d}{ d + \etanu q}}\,\sup_{\tau \ge t} \mathcal{F}[w(\tau)]^{\frac{1-m}{2-m} \frac{\etanu q}{d + \etanu q} }\,.
\end{aligned}
\]
This completes the proof of~\eqref{interp.Calpha.Lp} with $\varth=\tfrac{1-m}{2-m} $.\end{proof}

Now we compare the Fisher information with its linearized version in the spirit of~\cite[Lemma~7]{BBDGV} and~\cite[Lemma~5.1]{BGVm*}.
\begin{Lem}\label{Fisher-lin-nonlin}
Assume that $m\in(0,1)$. If~$w$  is a solution of~\eqref{eq.FDE.OU} corresponding to an initial datum $w_0$ satisfying assumptions {\rm (H1'')-(H2'')}, then
\be{eq: fish-lin}
\IL [w(t)] \le \frac{(1+h)^{3-2m}}{(1-m)\,(1-h)}\,\mathcal{I}[w(t)] + {\mu}_{h}\,h\,\FL[w(t)] \quad \textrm{for a.e.}~t >0\,,
\ee
where $\mu_h$ is such that
\[\textstyle
2\,(1-h)\,\mu_h=(2+\beta-\gamma)^2\,(2-m)^2\,(1-m)\,(1+h)^{4-2m}\,(1-4h) \left[ \frac12+\frac{2\,(3-m)}{3\,(1-h)^{4-m}}\,h \right]^2\,.
\]
\end{Lem}
\begin{proof} The proof is similar to the one of~\cite[Lemma~5.1]{BGVm*}: here we give some details for the reader's convenience. For the sake of greater readability we shall again omit time dependence.

To begin with, let us rewrite the Fisher information $\mathcal I$ as
\[\label{eq: fish-re}
\mathcal{I}[w] :=(1-m)^2 \int_{\RR^d} w \left|\nabla\big( A(w)\,(w-1)\,\mathfrak B^{m-1} \big)\right|^2 \mathfrak B\,\frac{dx}{|x|^\beta}\,,
\]
where we have set
\[\label{A}
A(w):=\frac{w^{m-1}-1}{(m-1)\,(w-1)}=:\frac{a(w)}{w-1}\,.
\]
It is easy to check that $\lim_{w \to 1} A(w)=1$, $A(w)>0$ and $A(w)\to 0$ as $w\to\infty$. Moreover,
\[\label{A'}
A'(w)=\frac{w^{m-2}-A(w)}{w-1}\le 0\,,
\]
since the function $a(w)$ is concave in $w$, so that its incremental quotient $A(w)$
(evaluated at $w=1$) is a nonincreasing function of $w$. In particular,
\[\label{bounds.A}
(1+h)^{m-2} \le A(w) \le (1-h)^{m-2}\,.
\]
Similarly, it is straightforward to show that $ A'(w) $ is bounded. Now let us set $g=(w-1)\mathfrak B^{m-1}$. Since $(w-1)A'(w)+A(w)=w^{m-2}$, we get:
\[\label{eq: dev-A}
\begin{aligned}
 \nabla\left( A(w)\,(w-1)\,\mathfrak B^{m-1} \right) &=A(w)\,\nabla g + A'(w)\,g\,\mathfrak B^{1-m}\,\nabla g
 + A'(w)\,g^2\,\nabla\left( \mathfrak B^{1-m} \right) \\
 &=\left[A(w)+ A'(w)\,(w-1)\right] \nabla g + A'(w)\,g^2\,\nabla\left( \mathfrak B^{1-m} \right) \\
 &=w^{m-2}\,\nabla g + A'(w)\,g^2\,\nabla\left( \mathfrak B^{1-m} \right).
\end{aligned}
\]
Using Young's inequality $ a\,b \le h\,a^2 + b^2 / 4h $ (for all $a$, $b \in \RR $) and the bounds $ 1-h \le w \le 1+h$, we get:
\[
\begin{aligned}
\frac{\mathcal{I}[w]}{(1-m)^2}=& \int_{\RR^d} w \left| w^{m-2}\,\nabla g + A'(w)\,g^2\,\nabla\left( \mathfrak B^{1-m} \right) \right|^2 \mathfrak B\,\frac{dx}{|x|^\beta} \\
 \ge&\,(1-h)\int_{\RR^d} \left|\nabla g \right|^2 w^{2m-3}\,\mathfrak B\,\frac{dx}{|x|^\beta} \\
 & -\,\frac{1-4\,h}{4\,h} \int_{\RR^d} g^4 \left| A'(w) \right|^2 w \left| \nabla\left( \mathfrak B^{1-m} \right) \right|^2 \mathfrak B\,\frac{dx}{|x|^\beta} \\
 \ge&\,\frac{1-h}{(1+h)^{3-2m}} \int_{\RR^d} \left|\nabla g \right|^2 \mathfrak B\,\frac{dx}{|x|^\beta} \\
&\,-\,\frac{(1+h)\,(1-4\,h)}{4\,h} \int_{\RR^d} g^4 \left| A'(w) \right|^2 \left| \nabla\left( \mathfrak B^{1-m} \right) \right|^2 \mathfrak B\,\frac{dx}{|x|^\beta}
\end{aligned}
\]
(in the last passage we have used the fact that $ h < 1/4 $). We have therefore established the inequality
\begin{multline*}\label{step.1.lin.nonlin}
\IL[g] \le\,\frac{(1+h)^{3-2m}}{(1-m)\,(1-h)}\,\mathcal{I}[w] \\
\,+ \frac{(1-m)\,(1+h)^{4-2m}(1-4\,h)}{4\,h\,(1-h)} \int_{\RR^d} g^4 \left| A'(w) \right|^2 \left| \nabla\left( \mathfrak B^{1-m} \right) \right|^2 \mathfrak B\,\frac{dx}{|x|^\beta}\,.
\end{multline*}
To complete the proof, it is enough to establish the inequality
\[\label{step.2.lin.nonlin}
\int_{\RR^d} g^4 \left| A'(w) \right|^2 \left| \nabla\left( \mathfrak B^{1-m} \right) \right|^2 \mathfrak B\,\frac{dx}{|x|^\beta}
\le Q \int_{\RR^d} g^2\,\mathfrak B^{2-m}\,\frac{dx}{|x|^\gamma}
\]
with $Q :=(2+\beta-\gamma)^2\,(2-m)^2 \left[ \frac12+\frac{2\,(3-m)}{3\,(1-h)^{4-m}}\,h \right]^2 h^2$. To this end, we observe that
\[\label{step.2.lin.nonlin.1}
\left| \nabla\left( \mathfrak B^{1-m} \right) \right|^2 \frac{\mathfrak B}{|x|^\beta}=\frac{(2+\beta-\gamma)^2\,|x|^{2+\beta-2\gamma}}{\left(C+|x|^{2+\beta-\gamma}\right)^4}\,\mathfrak B \le \frac{(2+\beta-\gamma)^2}{|x|^\gamma}\,\mathfrak B^{4-3m}\,,
\]
so that
\[\label{step.2.lin.nonlin.2}
\int_{\RR^d} g^4 \left| A'(w) \right|^2 \left| \nabla\left( \mathfrak B^{1-m} \right) \right|^2 \mathfrak B\,\frac{dx}{|x|^\beta} \le (2+\beta-\gamma)^2 \int_{\RR^d} g^4 \left| A'(w) \right|^2 \mathfrak B^{4-3m}\,\frac{\dx}{|x|^\gamma}\,.
\]
By definition of $g=(w-1)\,\mathfrak B^{m-1}$, using Taylor expansions and the bounds on $w$, through elementary computations we deduce that
\[\label{step.2.lin.nonlin.3}
g^2 \left| A'(w) \right|^2 \le \mathfrak B^{2m-2}\,(2-m)^2 \left[ \frac12+\frac{2\,(3-m)}{3\,(1-h)^{4-m}}\,h \right]^2 h^2\,,
\]
which concludes the proof.\end{proof}

\subsubsection*{$\bullet$ Convergence with sharp rates}\label{sect: conv-rates}
By means of the results of Section~\ref{sect: lin-nonlin} we shall first obtain a global (namely involving $\mathcal{F} $ and $\mathcal{I} $) inequality of Hardy-Poincar\'e type and then use it to get sharp rates of convergence for $\mathcal{F}[w(t)]$, which in turn will yield rates for the relative error in view of Lemma~\ref{lem.EntrVsLp}.
\begin{Lem}\label{lem: conv-entro} Assume that $m\in(0,1)$, $ m \neq m_\ast$. If~$w$  is a solution of~\eqref{eq.FDE.OU} corresponding to an initial datum $w_0$ satisfying assumptions {\rm (H1'')-(H2'')}, then there holds
\be{thm.entr.prod.ineq}
\left[ 2\,(1-m)\,\Lambda -\,\rho_h\,h \right] \mathcal{F}[w(t)] \le \frac{m}{1-m}\,\mathcal{I}[w(t)] \quad \textrm{for a.e.}~t>0\,,
\ee
where $\Lambda $ is the best constant appearing in the Hardy-Poincar\'e inequality~\eqref{HP-ineq},
\[\label{k1h}
\rho_h :=4\,(1-m)\,(3-m)\,\Lambda + \frac{(1-h)^{3-m}}{(1+h)^{3-2m}}\,{\mu}_{h}
\]
and $\mu_h $ is the same quantity as in Lemma~\ref{Fisher-lin-nonlin}.
\end{Lem}
\begin{proof} With no loss of generality we can assume that $\mathcal{F}[w(t)] \neq 0$ (and so also $\FL[w(t)] \neq 0$ thanks to Lemma~\ref{Lem.Bounds.RE}), otherwise there is nothing to prove. The Hardy-Poincar\'e inequality~\eqref{HP-ineq} plus Lemmas~\ref{Lem.Bounds.RE} and~\ref{Fisher-lin-nonlin} then yield
\[\label{HP.Step.1}
2\,(1-m)\,\Lambda \le \frac{\IL[w(t)]}{\FL[w(t)]} \le \frac{m\,(1+h)^{3-2m}}{(1-m)\,(1-h)^{3-m}}\,\frac{\mathcal{I}[w(t)]}{\mathcal{F}[w(t)]} + {\mu}_{h}\,h\,,
\]
which reads
\[\label{HP.Step.2}
\frac{(1-h)^{3-m}}{(1+h)^{3-2m}} \left[ 2\,(1-m)\,\Lambda - {\mu}_{h}\,h \right] \le \frac{m}{1-m}\,\frac{\mathcal{I}[w(t)]}{\mathcal{F}[w(t)]}\,.
\]
Finally, since
\[
\frac{\rd}{\rd h}\,\frac{(1-h)^{3-m}}{(1+h)^{3-2m}}=-\,\frac{(1-h)^{2-m}}{(1+h)^{4-2m}}\,(6-3\,m+m\,h) \ge -\,2\,(3-m)\,,
\]
so that
\[
\frac{(1-h)^{3-m}}{(1+h)^{3-2m}} \ge 1 - 2\,(3-m)\,h\,,
\]
we can deduce that
\begin{multline*}
\frac{(1-h)^{3-m}}{(1+h)^{3-2m}} \left[ 2\,(1-m)\,\Lambda - {\mu}_{h}\,h \right] \\
\ge\,2\,(1-m)\,\Lambda -\,\left[4\,(1-m)\,(3-m)\,\Lambda + \frac{(1-h)^{3-m}}{(1+h)^{3-2m}}\,{\mu}_{h} \right] h\,.
\end{multline*}
This concludes the proof.
\end{proof}

\medskip\begin{proof}[Proof of Theorem~\ref{Thm:Asymptotic rates}] Let $ m \neq m_\ast$ and assume that $w$ is a solution of~\eqref{eq.FDE.OU} corresponding to an initial datum $w_0$ satisfying assumptions {\rm (H1'')-(H2'')}. We have to prove that, for some constants $ \mathcal K_0,t_0 > 0$ that depend on $d$, $m$, $\gamma$, $\beta$, $ C_1$, $ C$, $ C_2 $ and $w_0$, the decay estimate
\be{eq: estimate-sharp-entr}
\mathcal{F}[w(t)] \le \mathcal K_0\,e^{-\,2\,(1-m)\,\Lambda\,(t-t_0)} \quad \forall\,t \ge t_0
\ee
holds. We split the proof in two steps: in the first one we provide a non-sharp exponential decay for $\mathcal{F}[w(t)]$, in the second one we use the latter to get the sharp rate. We adopt implicitly the same notations as in Lemma~\ref{lem: conv-entro}.

\smallskip\noindent$\bullet~$\textsc{Step 1.}  By Proposition~\ref{prop: cvrgcewithoutrate} we know that $ h(t):=\|w(t)-1\|_{\infty} \to 0$ as $ t \to \infty$. According to Lemma~\ref{Lem.Bounds.RE}, there exists $ t_0 > 0$ such that $h(t)\le1/4$ for any $t\ge t_0$, and we can additionally require that
\[
\inf_{t \ge t_0} \left[ 2\,(1-m)\,\Lambda -\,\rho_{h(t)}\,h(t) \right] \ge (1-m)\,\Lambda\,.
\]
By combining this information,~\eqref{Entropy.Prod} and~\eqref{thm.entr.prod.ineq}, we obtain
\[\label{step1.entr.rates.1}
\frac{\rd}{\dt}\,\mathcal{F}[w(t)]=-\,\frac{m}{1-m}\,\I[w(t)] \le -\,(1-m)\,\Lambda\,\mathcal{F}[w(t)] \quad \textrm{for a.e.}~t > t_0\,,
\]
which yields the exponential-decay estimate
\[\label{step1.entr.rates.2}
\mathcal{F}[w(t)] \le \mathcal{F}[w(t_0)]\,{e}^{-(1-m)\,\Lambda\,(t-t_0)} \le \mathcal{F}[w_0]\,{e}^{-(1-m)\,\Lambda\,(t-t_0)} \quad \forall\,t \ge t_0\,.
\]

\smallskip\noindent$\bullet~$\textsc{Step 2.}  As a consequence of Lemma~\ref{lem.EntrVsLp} and in particular~\eqref{interp.Calpha.Lp}, we can infer that
\[\label{step2.entr.rates.1}
\sup_{\tau \ge t} h(\tau) \le \kappa_\infty\,\mathcal{F}[w_0]^{\varth}\,{e}^{-\,\varth\,(1-m)\,\Lambda\,(t-t_0)} \quad \forall\,t \ge t_0\,.
\]
Moreover, it is clear that
\[\label{step2.entr.rates.1-bis}
0 < \rho_\infty :=\sup_{t \ge t_0} \rho_{h(t)} < \infty\,.
\]
hence, inequality~\eqref{thm.entr.prod.ineq}, which also holds with $h=h(t)$, implies
\[
\left[ 2\,(1-m)\,\Lambda -\,\rho_\infty\,\kappa_\infty\,\mathcal{F}[w_0]^{\varth}\,{e}^{-\,\varth\,(1-m)\,\Lambda\,(t-t_0)} \right] \mathcal{F}[w(t)] \le \frac{m}{1-m}\,\mathcal{I}[w(t)]
\]
for a.e.~$ t>t_0$, so that by using again~\eqref{Entropy.Prod} we end up with the differential inequality
\[\label{diff-ineq-iter}
\frac{\rd}{\dt}\,\mathcal{F}[w(t)] \le -\,\left[ 2\,(1-m)\,\Lambda -\,\rho_\infty\,\kappa_\infty\,\mathcal{F}[w_0]^{\varth}\,{e}^{-\,\varth\,(1-m)\,\Lambda\,(t-t_0)} \right] \mathcal{F}[w(t)]
\]
for a.e.~$ t>t_0$. An explicit integration then gives
\[\label{diff-ineq-iter-int}
\mathcal{F}[w(t)] \le \mathcal{F}[w_0]\,e^{\frac{\rho_\infty\,\kappa_\infty\,\mathcal{F}[w_0]^{\varth}}{\varth\,(1-m)\,\Lambda} \left[ 1 - e^{-\,\varth\,(1-m)\,\Lambda\,(t-t_0)} \right] }\,e^{-\,2\,(1-m)\,\Lambda\,(t-t_0)}
\]
for all $t \ge t_0$, namely~\eqref{eq: estimate-sharp-entr} with $ \mathcal K_0 :=\mathcal{F}[w_0]\,e^{\frac{\rho_\infty\,\kappa_\infty\,\mathcal{F}[w_0]^{\varth}}{\varth\,(1-m)\,\Lambda}}$.
\end{proof}

Theorem~\ref{Thm:RUC} follows as a straightforward consequence of Theorem~\ref{Thm:Asymptotic rates}, Lemma~\ref{lem.EntrVsLp} and standard interpolation.

\section{Additional results and comments}\label{Sec:Additional}

\subsection{Best matching, refined estimates and \texorpdfstring{$\LL^{1,\gamma}$}{L1}-convergence}\label{Sec:Refinements}

The \emph{relative entropy to the best matching Barenblatt function} is defined as
\[
\mathcal G[v]:=\inf_{\mu>0}\frac1{m-1}\int_{\R^d}\left[v^m-\mathfrak B_\mu^m-m\,\mathfrak B_\mu^{m-1}\,(v-\mathfrak B_\mu) \right] \frac{\rd x}{|x|^\gamma}\,,
\]
where the optimization is taken with respect to the scaling parameter $\mu>0$, that is, with respect to the set of the scaled Barenblatt functions
\[
\mathfrak B_\mu(x):=\mu^{d-\gamma}\,\mathfrak B(\mu\,x)\quad\forall\,x\in\R^d\,.
\]
We start by a computation of the asymptotic rates which follows the line of thought developed in~\cite{1004,1751-8121-48-6-065206}. Also see~\cite{MR1491842} for earlier considerations in this direction. An elementary calculation shows that in fact
\be{G}
\mathcal G[v]=\frac1{m-1}\int_{\R^d}\left[ v^m-\mathfrak B_{\mu_\star}^m\right] \frac{\rd x}{|x|^\gamma}\,,
\ee
where $\mu_\star$ is the unique scaling parameter for which
\be{Eqn:MomentCondition}
\int_{\R^d}|x|^{2+\beta-\gamma}\,v\,\frac{\rd x}{|x|^\gamma}=\int_{\R^d}|x|^{2+\beta-\gamma}\,\mathfrak B_{\mu_\star}\,\frac{\rd x}{|x|^\gamma}=\mu_\star^{-(2+\beta-\gamma)}\int_{\R^d}|x|^{2+\beta-\gamma}\,\mathfrak B\,\frac{\rd x}{|x|^\gamma}\,.
\ee
This approach can be applied to any function $v\in\LL^{1,\gamma}(\R^d)$ and in particular to a $t$-dependent solution to~\eqref{FD-FP}. Moreover, we observe that
\[
\frac \rd{\rd t}\int_{\R^d}|x|^{2+\beta-\gamma}\,v(t,x)\,\frac{\rd x}{|x|^\gamma}=-\,(2+\beta-\gamma)\,\frac{(1-m)^2}m\,\mathcal G[v(t)]\le 0\,.
\]
Hence $\mu_\star=\mu_\star(t)$ is monotone, with a positive limit as $t\to\infty$, and this limit has to be equal to $1$. Another remark is that
\[
\frac \rd{\rd t}\,\mathcal G[v(t)]=-\,\frac m{1-m}\,\mathcal J[v(t)]\,,
\]
where $\mathcal J[v]$ denotes the \emph{relative Fisher information with respect the best matching Barenblatt function}, defined as
\[
\mathcal J[v]:=\int_{\R^d}v\left|\,\nabla v^{m-1}-\nabla\mathfrak B_{\mu_\star}^{m-1}\right|^2\,\frac{\rd x}{|x|^\beta}\,.
\]
We can consider the linearized regime: if $v=\mathfrak B_{\mu_\star}\,(1+\varepsilon\,\mathfrak B_{\mu_\star}^{1-m}\,f)$, by neglecting higher order terms in $\varepsilon$, the moment condition~\eqref{Eqn:MomentCondition} becomes
\be{Eqn:MomentConditionLinearized}
\int_{\R^d}|x|^{2+\beta-\gamma}\,\mathfrak B_{\mu_\star}^{2-m}\,f\,\frac{\rd x}{|x|^\gamma}=0\,.
\ee
Let us recall the parameter $\rho$ defined for the self-similar solution of the introduction by $\frac1\rate=(d-\gamma)\,(m-m_c)$ with $m_c=\tfrac{d-2-\beta}{d-\gamma}$. With a simple scaling, we can also note that the spectral gap inequality of Proposition~\ref{Prop:Spectrum} is changed into
\[
\int_{\R^d}|\nabla f|^2\,\mathfrak B_{\mu_\star}\,\frac{\rd x}{|x|^\beta}\ge\Lambda\,\mu_\star^\frac1\rho\int_{\R^d}|f|^2\,\mathfrak B_{\mu_\star}^{2-m}\,\frac{\rd x}{|x|^\gamma}
\]
for any $f\in\mathrm L^2(\R^d,\mathfrak B_{\mu_\star}^{2-m}\,|x|^{-\gamma}\,\rd x)$ such that $\int_{\R^d}f\,\mathfrak B_{\mu_\star}^{2-m}\,|x|^{-\gamma}\,\rd x=0$ if $m>m_\ast$ and~\eqref{Eqn:MomentConditionLinearized} holds. However, compared to Proposition~\ref{Prop:Spectrum}, we obtain that the inequality holds with $\Lambda=\Lambda_{\rm ess}$ if $\delta\le(n+2)/2$ and with $\Lambda=\Lambda_{0,1}$ if $\delta\ge n/(2-\nueta)$, but with an improved spectral gap $\Lambda>\Lambda_{1,0}$ if $(n+2)/2<\delta<n/(2-\nueta)$, because of the orthogonality condition~\eqref{Eqn:MomentConditionLinearized}. See~\cite[Appendix~B]{BDMN2016a} for details. Hence, by arguing as for the proof of Theorem~\ref{Thm:Asymptotic rates}, we obtain for the relative entropy $\mathcal G$ the following improved convergence rate.
\begin{Prop}\label{Prop:Asymptotic rates} Let $d\ge2$ and assume that~\eqref{parameters-1} holds, $m\in(0,1)$, $m\neq m_\ast$. If $m\in(0,m_\ast)$, we assume that $(v_0-\mathfrak B)\in\mathrm L^{1,\gamma}(\R^d)$, while we choose $C=C(M)$ if $m>m_\ast$. With same notations as in Proposition~\ref{Prop:Spectrum}, if $v$ solves~\eqref{FD-FP} and~\eqref{Ineq:sandwiched} holds, then there exists a positive constant~$\mathcal C$ such that
\[\label{AsymptoticEntropyDecay-bis}
\mathcal G[v(t)]\le\mathcal C\,e^{-\,2\,(1-m)\,\min\{\Lambda_{\rm ess},\Lambda_{0,1}\}\,t}\quad \forall\,t \ge0\,.
\]
\end{Prop}

Next, we adapt the Csisz\'ar-Kullback-Pinsker inequality of~\cite{DT2011} to our setting. We recall that $\widetilde m_1:=\frac{d-\gamma}{d+2+\beta-2\,\gamma}$.
\begin{Lem}\label{Lem:CKP} Let $d\ge 1$, $m\in(\widetilde m_1,1)$ and assume that~\eqref{parameters-1} holds. If $v$ is a non-negative function in $\LL^{1,\gamma}(\R^d)$ such that $\mathcal G[v]$ is finte. If $\nrm v{1,\gamma}=M$, then
\[
\mathcal G[v]\ge\frac m{8\,\nrm{\mathfrak B_{\mu_\star}^m}{1,\gamma}^m}\(C(M)\,\nrm{v-\mathfrak B_{\mu_\star}}{1,\gamma}+\int_{\R^d}|x|^{2+\beta-\gamma}\left|v-\mathfrak B_{\mu_\star}\right|\frac{\rd x}{|x|^\gamma}\)^2\,.
\]
\end{Lem}
The proof goes exactly along the lines of the one of~\cite[Theorem~4]{DT2011}, except that the expression of $\mathfrak B_{\mu_\star}$ and the weight $|x|^{-\gamma}$ have to be taken into account. Details are left to the reader. Proposition~\ref{Prop:Asymptotic rates} and Lemma~\ref{Lem:CKP} can be combined to give the result of convergence in $\LL^{1,\gamma}(\R^d)$ stated in Theorem~\ref{Thm:Norms}.

\subsection{Optimality of the constant on the curve of Felli and Schneider}\label{Sec:EPOpbFS}

For completeness, let us give the key idea of the proof of~Theorem~\ref{Thm:BDMN-I}, (i), since the framework of the functional $\mathcal G$ is well adapted. In~\cite{BDMN2016a}, the proof is purely variational, but the flow setting is particularly convenient as we shall see next.
\begin{Lem}\label{Lem:ImprovedScalings} Under the assumptions of Theorem~\ref{Thm:BDMN-I}, there exists a convex function~$\Phi$ with $\Phi(0)=0$ and $\Phi'(0)=\tfrac{1-m}m\,(2+\beta-\gamma)^2$ such that
\[
\mathcal J[v]\ge\Phi\big(\mathcal G[v]\big)\,.
\]
\end{Lem}
\begin{proof} As in~\cite[Proposition~7]{BDMN2016a}, we notice that
\begin{multline*}
\mathcal J[v]-\tfrac{1-m}m\,(2+\beta-\gamma)^2\,\mathcal G[v]\\
\textstyle=\frac2{\alpha\,p}\,\frac{(m-1)^2}{(2\,m-1)^2}\(\mathrm a\,\nrm{\D w}{2,d-n}^2+\mathrm b\,\nrm w{p+1,d-n}^{p+1}-\nrm w{2p,d-n}^{2\,p\,\frac{n+2-p(n-2)}{n-p(n-4)}}\)
\end{multline*}
for some explicit constants $\mathrm a$ and $ \mathrm b$ and for $w=v^{m-\frac12}$. For a given function $w\in C_0^\infty(\R^d)$, let us consider $w_\scaling(x):=\scaling^\frac n{2p}\,w(\scaling\,x)$ for any $x\in\R^d$. An optimization with respect to $\scaling$ as in~\cite{dolbeault:hal-01081098} shows the existence of a convex function $\Psi$ such that
\[
\mathcal J[v]-\tfrac{1-m}m\,(2+\beta-\gamma)^2\,\mathcal G[v]\ge\Psi\Big(\tfrac1{m-1}\big(\nrm w{p+1,d-n}^{p+1}-\nrm{\mathfrak B_{\mu_\star}^{m-\frac12}}{p+1,d-n}^{p+1}\big)\Big).
\]
The conclusion holds with $\Phi(s)=\tfrac{1-m}m\,(2+\beta-\gamma)^2\,s+\Psi(s)$ using \eqref{G} and $\nrm w{p+1,d-n}^{p+1}=\int_{\R^d}v^m\,|x|^{-\gamma}\,dx$. An elementary computation shows that $\Psi$ is convex with $\Psi(0)=\Psi'(0)=0$.\end{proof}

\begin{proof}[Proof of~Theorem~\ref{Thm:BDMN-I}, (i)] Under the assumptions of Theorem~\ref{Thm:Asymptotic rates}, it is clear that the optimality in the inequality $\mathcal J[v]\ge\tfrac{1-m}m\,(2+\beta-\gamma)^2\,\mathcal G[v]$ for a solution $v=v(t)$ to~\eqref{FD-FP} can be achieved only in the asymptotic regime, hence showing that $\Lambda_\star=\tfrac12\,(2+\beta-\gamma)^2/(1-m)\ge\Lambda_{0,1}$. On the other hand, if symmetry holds in~\eqref{CKN}, the opposite inequality also holds and hence we have equality. This characterizes the curve $\beta=\beta_{F\rm S}(\gamma)$. This proof of course holds only for solutions corresponding to initial data such that~\eqref{Ineq:sandwiched} is satisfied, but an appropriate regularization allows us to conclude in the general case.\end{proof}

\subsection{Concluding remarks}\label{Sec:Conclusion}

When $(\beta,\gamma)=(0,0)$, we know from~\cite{MR1940370,BBDGV} that
\[
\frac m{1-m}\,\mathcal K(M)=2\,(1-m)\,\Lambda_\star\quad\mbox{with}\quad\Lambda_\star=\Lambda_{0,1}\,,
\]
so that the global rate is the same as the asymptotic one obtained by linearization, and the corresponding eigenspace can be identified by considering the translations of the Barenblatt profiles. When $(\beta,\gamma)\neq(0,0)$, we may wonder when $\Lambda_\star=\Lambda_{0,1}$. Using the results of~\cite[Lemma~8]{BDMN2016a}, we can deduce that this holds whenever $\nueta=1$, which means $\alpha=\alpha_{\rm FS}$ or, equivalently, $\beta=\beta_{\rm FS}(\gamma)$.

As mentioned in the Introduction, in the case $(\beta,\gamma)=(0,0)$ Theorem~\ref{Thm:RUC} provides a better rate of convergence for the relative error with respect to the one obtained in~\cite{BBDGV}; in particular, we have the same rate for all $\LL^q $ norms with $ q \in \big[\tfrac{2-m}{1-m},\infty\big]$. However, in the case $\gamma \in (0,d) $ the rate in~\eqref{AsymptoticDecay.RelErr} still depends on $q$. To some extent, this has to be expected. Indeed, as soon as $\gamma >0$, it can easily be shown that Gagliardo-Nirenberg interpolation inequalities of the type of~\eqref{GN-interp} \emph{fail} if in the right-hand side one puts an $\LL^{p,\gamma} $ norm. Since such inequalities are key in order to turn the decay of the free energy~\eqref{eq: estimate-sharp-entr} into a \emph{uniform} decay, the only way we can exploit them, as it is clear from the proof of Lemma~\ref{lem.EntrVsLp}, is by bounding a non-weighted norm of the relative error with a weighted norm or the free energy, like in~\eqref{Entr.Lp.gamma-2}, and this is precisely what causes the rate to differ.

Finally, let us mention a puzzling moment conservation. It is straightforward to check that
\[
\frac \rd{\rd t}\int_{\R^d}x\,|x|^\beta\,v\,\frac{\rd x}{|x|^\gamma}=0\,.
\]
This moment corresponds to the eigenfunction $f_{0,1}$, up to a multiplication by a constant, if $\beta+1=\alpha\,\nueta$. The reader is invited to check that this is possible if and only if $\beta=0$. See~\cite[Appendix~B]{BDMN2016a} for technical details. When $(\beta,\gamma)=(0,0)$, the lowest moments are clearly associated with eigenspaces of the linearized evolution operator and responsible for the asymptotic rates of convergence of the evolution equation. If $(\beta,\gamma)\neq(0,0)$, the interpretation is not as straightforward.

\bigskip\begin{center}\rule{4cm}{0.5pt}\end{center}\medskip\appendix

\appendix\section*{Appendix. H\"older regularity at the origin for a degenerate/singular linear problem}\label{App: c-alfa}
First of all we observe that, to our purposes, it is convenient to change variables as in~\cite[Section~3.3]{BDMN2016a}, so that $v(t,r,\omega)=z(t,s,\omega)$ with $s=r^\alpha$  transforms~\eqref{FD-FP} into
\[
z_t-\,\mathsf D_\alpha^* \left[z\,\D \big( z^{m-1}-|x|^2 \big) \right]=0
\]
upon defining $\mathsf D_\alpha^*$ as the adjoint to $\D$ on $\L^2(\R^d,|x|^{n-d}\,\rd x)$, where the parameters $\alpha $ and $ n $ are as in~\eqref{parameters} and
\[
\mathsf D_\alpha z :=\left( \alpha\,\frac{\partial z}{\partial s}\,, \frac1s\,\nabla_\omega\,z \right)\,.
\]
In this regard, let us recall here some basic facts taken from~\cite[Section~3.3]{BDMN2016a}. If $\mathbf f$ and $g$ are respectively a vector-valued function and a scalar-valued function, then
\[
\int_{\R^d} \mathbf f \cdot (\D g)\,|x|^{n-d}\,\rd x=\int_{\R^d} (\mathsf D_\alpha^* \mathbf f)\,g\,|x|^{n-d}\,\rd x\,.
\]
In other words, if we take a representation of $\mathbf f$ adapted to spherical coordinates, that is $s=|x|$ and $\omega=x/s$, and consider $f_s:=\mathbf f\cdot \omega$ and $\mathbf f_\omega:=\mathbf f-f_s\,\omega$, then
\[
\mathsf D_\alpha^*\mathbf f=-\,\alpha\,s^{1-n}\,\frac\partial{\partial s}\big(s^{n-1}\,f_s\big)-\,\frac1s\,\nabla_{\!\omega}\cdot\mathbf f_\omega\,,
\]
where $\nabla_{\!\omega}$ denotes the gradient with respect to angular derivatives only. In particular,
\[
\mathsf D_\alpha^* \left[z_1\,\D z_2 \right]=-\,\D z_1 \cdot \D z_2 + z_1\,\mathsf D_\alpha^*\(\D z_2\)
\]
with
\[
-\,\mathsf D_\alpha^*\(\D z_2\)=\frac{\alpha^2}{s^{n-1}}\,\frac\partial{\partial s}\(s^{n-1}\,\frac{\partial z_2}{\partial s}\)+\frac1{s^2}\,\Delta_\omega z_2
\]
where $\Delta_\omega$ represents the Laplace-Beltrami operator acting on $\omega\in\S^{d-1}$.

\medskip The advantage of resorting to this change of variables is that we can transform a problem with two different weights $ |x|^{-\gamma}$ and $ |x|^{-\beta} $ into a problem with two weights that are equal to $ |x|^{n-d}$. It is remarkable that Barenblatt-type stationary solutions~\eqref{Barenblatt-intro2} are transformed into the standard Barenblatt profiles
\[
\mathcal B=\(C+|x|^2\)^\frac1{m-1}\quad\forall\,x\in\RR^d\,.
\]
Details on the change of variables can be found in~\cite[Section~2.3]{BDMN2016a}. With regards to the purpose of this Appendix, the main interest of the change of variables is that it allows to use standard intrinsic cylinders. Given $ (t_0,x_0) \in \RR^+ \times \RR^d $ and $r>0$, let
\[\label{def.intr.cyl}
\begin{aligned}
Q_r(t_0,x_0) & :=\left\{(t,x) \in \RR^+ \times \RR^d :\,t_0-2\,r^2 <t<t_0\,, \ |x-x_0|<2\,r \right\}, \\
Q^+_r(t_0,x_0) &:=\left\{(t,x) \in \RR^+ \times \RR^d :\,t_0-\tfrac14\,r^2 < t < t_0\,,\ |x-x_0|<\tfrac12\,r \right\}, \\
Q^-_r(t_0,x_0) &:=\left\{(t,x)\in \RR^+ \times \RR^d :\,t_0-\tfrac78\,r^2 < t < t_0-\tfrac58\,r^2\,, \ |x-x_0|<\tfrac12\,r \right\}.
\end{aligned}
\]
\begin{figure}[ht]
\hspace*{-6pt}\includegraphics[width=6.25cm]{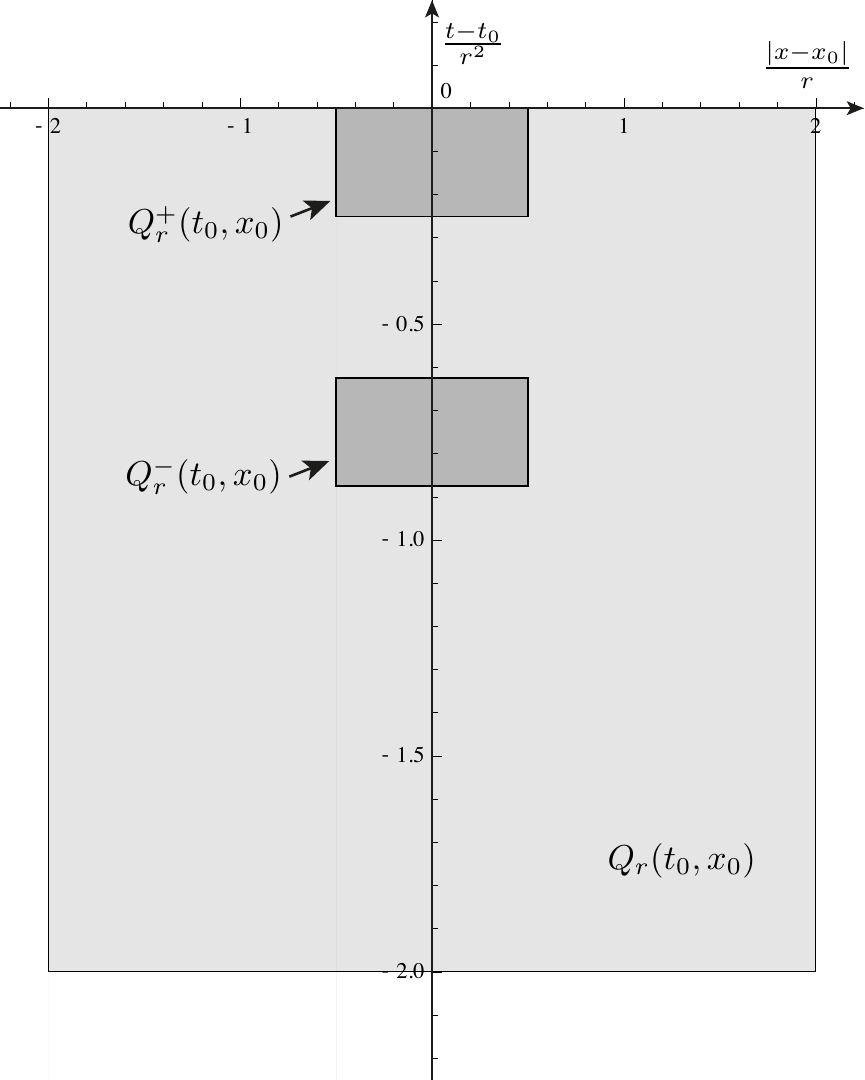}
\caption{\label{Fig2} The intrinsic cylinders $Q_r(t_0,x_0)$, $Q^+_r(t_0,x_0)$ and $Q^-_r(t_0,x_0)$.}
\end{figure}
See Fig.~\ref{Fig2}. As a straightforward consequence of the above definitions, there holds $Q_{r/4}(t_0,x_0)\subset Q^+_r(t_0,x_0)$. The above cylinders are the same as the classical parabolic cylinders: having same weights gives the same scaling properties as in the non-weighted case, as first remarked in~\cite{ChSe85}.

\medskip Our aim here is to study the local H\"older regularity for solutions to a weighted linear problem of the form
\be{eq.lin.ab}
u_t+\,\mathsf D_\alpha^*\Big[\,\mathsf a\,(\D u+\mathsf B\,u) \Big]=0\quad\mbox{in}\quad\RR^+ \times \RR^d
\ee
for some functions $\mathsf a$ and $\mathsf B$ which depend on $(t,x)\in\RR^+\times\RR^d$. By following the ideas of F.~Chiarenza and R.~Serapioni in~\cite{ChSe85}, we start by establishing a parabolic Harnack inequality, through a weighted Moser iteration.
\begin{Prop}[A parabolic Harnack inequality]\label{Thm.Harnack.lin} Assume that $\mathsf a$ is locally bounded and bounded away from zero and that $\mathsf B$ is locally bounded in $\RR^+\times\RR^d$. Let $d\ge2$, $\alpha>0$ and $n>d$. If $u$ is a bounded positive solution of~\eqref{eq.lin.ab}, then for all $(t_0,x_0)\in\RR^+\times\RR^d$ and $r>0$ such that $Q_r(t_0,x_0)\subset\RR^+\times B_1$, we have
\[\label{Thm.Harnack.lin.1}
\sup_{Q_r^-(t_0,x_0)} u \le H \inf_{Q_r^+(t_0,x_0)} u\,.
\]
The constant $H>1$ depends only on the local bounds on the coefficients $\mathsf a$, $\mathsf B$ and on $d$, $\alpha$, and $n$.\end{Prop}
\begin{proof} The proof follows the lines of~\cite[Theorem~2.1]{ChSe85} with minor modifications. Let us emphasize the main adaptations. We observe that a critical Caffarelli-Kohn-Nirenberg inequality can be rewritten after the change of variables $s=r^\alpha$ as
\[
\(\int_{\R^d}|w|^\frac{2\,n}{n-2}\,|x|^{n-d}\,dx\)^\frac{n-2}n\le\mathsf K_{n,\alpha}\int_{\R^d}|\mathsf D_\alpha w|^2\,|x|^{n-d}\,dx\quad\forall\,w\in C^\infty_0(\RR^d)
\]
and is actually scale invariant. See~\cite[Inequality~3.2]{DEL2015} for details, including symmetry issues and the computation of $\mathsf K_{n,\alpha}$ in the symmetry range. This inequality plays the same role as the one of~\cite[Lemma~1.1]{ChSe85}. Then the proof follows upon replacing $\nabla$ by $\D$. The term $\mathsf B\,u$ is in fact of lower order, since it is locally bounded: it can easily be reabsorbed into the energy estimates. By translating the intrinsic cylinders with respect to~$t$ by $r^2$, we achieve the conclusion.\end{proof}

The Harnack inequality of Proposition~\ref{Thm.Harnack.lin} implies a H\"older continuity, by adapting the classical method \emph{\`a la De Giorgi} to our weighted framework.
\begin{Cor}[H\"older regularity at the origin I]\label{Thm.C.alpha.lin}
Under the same assumptions as in Proposition~\ref{Thm.Harnack.lin}, there exist $\etanu\in (0,1)$ and $\mathcal K>0$ such that
\[\label{Thm.C.alpha.lin.1}
|u(t)|_{C^\etanu(B_{1/2})} \le\mathcal K \left\| u \right\|_{\LL^\infty((t+2,t+3)\times B_4)} \quad \forall\,t \ge 1\,,
\]
where $\etanu$ and $\mathcal K$ depend only on the constant $H>1$ of Proposition~\ref{Thm.Harnack.lin}.
\end{Cor}
\begin{proof} We fix $t_0 \ge 1$, $ r \in (0,1/2) $ and denote for simplicity $Q_{r}:=Q_r(t_0,0)$ and $Q_r^{\pm}:=Q_r^{\pm}(t_0,0)$. Let us introduce the following quantities:
\[
M_r :=\sup_{Q_r} u\,, \quad M_r^{\pm}:=\sup_{Q^{\pm}_r} u\,, \quad m_r:=\inf_{Q_r} u\,, \quad m_r^{\pm}:=\inf_{Q^{\pm}_r}u\,.
\]
We apply Proposition~\ref{Thm.Harnack.lin} to the nonnegative solution $M_{2r}-u$ to obtain
\[\label{Thm.C.alpha.lin.2}
M_{2r}-m_{r}^-=\sup_{Q_{r}^-}(M_{2r}-u) \le H\,\inf_{Q_{r}^+}(M_{2r}-u)=H\,(M_{2r}-M^+_{r})\,.
\]
Similarly, by using $u-m_{2r}$ we obtain the inequality $M_{r}^- - m_{2r}\le H\,(m^+_{r}-m_{2r})$ which, summed up with the previous inequality, gives
\[\label{Thm.C.alpha.lin.3}
H\,(M^+_{r}-m^+_{r}) + M_{r}^- - m_{r}^-\le (H-1)\,(M_{2r}-m_{2r})\,.
\]
Notice that we can always assume that $H>1$. Using $Q_{r/4}\subset Q^+_r$, we conclude that
\[\label{Thm.C.alpha.lin.5}
\osc_{Q_{r/4}}u\le \osc_{Q^+_r}u=M^+_{r}-m^+_{r} \le \frac{H-1}{H}\,(M_{2r}-m_{2r})=\frac{H-1}{H}\,\osc_{Q_{2r}}u\,.
\]
Without loss of generality we can assume that $H/(H-1)\le 8$: a well-known iteration technique (see, \emph{e.g.},~\cite[Lemma~6.1]{Giusti}) then shows that
\[\label{Thm.C.alpha.lin.5b}
\osc_{Q_{r}}u\le C\,(2\,r)^{\etanu}\,\osc_{Q_{1/2}}u\quad\forall\,r\in(0,1/2]\,,
\]
with $\etanu:=\log({H}/(H-1))/\log8\in(0,1)$ and $C>0$ depending only on $H$. A standard covering argument thus yields uniform H\"older continuity on smaller cylinders, namely
\[\label{Thm.C.alpha.lin.6}
|u|_{C^{\etanu,\etanu/2}(Q_{r})} \le 2^\etanu\,K\,\|u\|_{\LL^\infty(Q_{1/2})}\quad\forall\,r\in(0,1/4]\,,
\]
where $K>0$ is another constant that depends only on $H$ and we set
\[
|u|_{C^{\etanu,\etanu/2}(Q_{r})}:=\sup_{(t,x),\,(\tau,y) \in Q_{r}} \frac{\left| u(t,x)-u(\tau,y) \right|}{(|x-y|^2+|t-\tau|)^{\etanu/2}}\,.
\]
In particular we deduce that
\[\label{Thm.C.alpha.lin.1-proof}
| u(t_0) |_{C^\etanu(B_{1/2})} \le K \left\| u \right\|_{\LL^\infty((t_0-1/2, t_0) \times B_1)}\,.
\]
Now note that, as a trivial consequence of Proposition~\ref{Thm.Harnack.lin} (just replace the $\inf $ with the $\sup $ in the r.h.s.), there holds
\[
\left\| u \right\|_{\LL^\infty((t_0-1/2, t_0) \times B_1 )} \le H \left\| u \right\|_{\LL^\infty((t+2,t+3)\times B_4)}\,,
\]
which concludes the proof with $\mathcal K=K\,H$.\end{proof}

Since the change of variables $s=r^\alpha$ transforms H\"older functions into H\"older functions (but of course not $C^1$), as a direct consequence of Corollary~\ref{Thm.C.alpha.lin} we have an analogous result for the original (linear) equation.
\begin{Cor}[H\"older regularity at the origin II]\label{Thm.C.alpha.lin-2} Assume that $d$, $\beta$ and $\gamma$ comply with~\eqref{parameters-1}. If $u$ is a bounded positive solution of
\[
|x|^{-\gamma}\,u_t=\nabla \cdot \left[ |x|^{-\beta}\,a(t,x) \left( \nabla u + B(t,x)\,u \right) \right] \quad\mbox{in}\quad\RR^+ \times \RR^d\,,
\]
where $ a $ is locally bounded and bounded away from $0$, $B$ is locally bounded in $\RR^+ \times \RR^d$. Then there exist $\etanu\in (0,1)$ and $\mathcal K>0$ such that
\[\label{Thm.C.alpha.lin.orig}
| u(t) |_{C^\etanu(B_{2^{-1/\alpha}})} \le \mathcal K \left\| u \right\|_{\LL^\infty ((t+2, t+3) \times B_{2^{2/\alpha}})} \quad \forall\,t \ge 1\,,
\]
where $\etanu$ and $\mathcal K$ depend only on the local bounds on the coefficients $a$, $B$ and on $d$, $\gamma$, $\beta$. Here $\alpha$ is given by~\eqref{parameters}.\end{Cor}

\section*{Acknowledgments} This research has been partially supported by the projects \emph{STAB} (J.D., B.N.) and \emph{Kibord} (J.D.) of the French National Research Agency (ANR). M.B.~has been funded by Project MTM2011-24696 and MTM2014-52240-P (Spain). This work has begun while M.B.~and M.M.~were visiting J.D.~and B.N.~in 2014. M.B.~thanks the University of Paris 1 for inviting him. M.M.~has been partially funded by the National Research Project ``Calculus of Variations'' (PRIN 2010-11, Italy) and by the ``Universit\`a Italo-Francese / Universit\'e Franco-Italienne'' (Bando Vinci 2013). J.D.~also thanks the University of Pavia for support.

\smallskip\noindent {\sl\small\copyright~2016 by the authors. This paper may be reproduced, in its entirety, for non-commercial purposes.}
\providecommand{\href}[2]{#2}
\providecommand{\arxiv}[1]{\href{http://arxiv.org/abs/#1}{arXiv:#1}}
\providecommand{\url}[1]{\texttt{#1}}
\providecommand{\urlprefix}{URL }

\end{document}